\documentclass[12pt]{amsart}

\usepackage{hyperref}
\usepackage{amsmath,amssymb,latexsym,amsthm}
\usepackage{url}
\usepackage[shortlabels]{enumitem}
\usepackage{graphicx,color}
\usepackage[margin=1in]{geometry}
\usepackage{xcolor}
\usepackage{float}
\usepackage{booktabs}
\usepackage{multirow}

\DeclareMathOperator{\supp}{supp}

\DeclareMathOperator{\Tr}{Tr}

\newtheorem{theorem}{Theorem}[section]
\newtheorem{corollary}[theorem]{Corollary}

\newtheorem{lemma}[theorem]{Lemma}
\newtheorem{prop}[theorem]{Proposition}

\theoremstyle{definition}
\newtheorem{defn}[theorem]{Definition}

\theoremstyle{remark}

\newcommand{\bpr}{\begin{proof}\hspace{3pt}}
\newcommand{\epr}{\end{proof}}

\newcommand{\N}{\mathbb{N}}
\newcommand{\Z}{\mathbb{Z}}
\newcommand{\R}{\mathbb{R}}
\newcommand{\C}{\mathbb{C}}
\newcommand{\eps}{\epsilon}
\newcommand{\pa}{\partial}
\newcommand{\lsim}{\lesssim}

\def\bra{\langle}
\def\ket{\rangle}
\begin{document}

\title{Reconstruction of Rough Conductivities from Boundary Measurements}
\author{Ashwin Tarikere}

\address{Department of Mathematics and Statistics, University of Jyv{\"a}skyl{\"a}}
\email{ashtarik@jyu.fi}




\maketitle

\begin{abstract}
	We show the validity of Nachman's procedure (\textit{Ann. Math.} 128(3):531--576, 1988)  for reconstructing a conductivity function $\gamma$ in a smooth bounded domain $\Omega \subset \R^n$ ($n\geq 3$) from its Dirichlet-to-Neumann map $\Lambda_\gamma$ for less regular conductivities, specifically $\gamma \in H^{3/2,2n}(\Omega)$ such that $\gamma \equiv 1$ near $\pa \Omega$. We also obtain a log-type stability estimate for the inverse problem when $\gamma$ has slightly higher regularity, i.e., $\gamma \in H^{2-s,n/s}(\Omega)$ for  $0 < s <1/2$.
\end{abstract}

\section{Introduction}

	Let $\Omega$ be a bounded domain in $\R^n \, (n \geq 3)$ with sufficiently smooth boundary, and let $\gamma$ be a positive real-valued function in $\Omega$ satisfying the uniform ellipticity condition
	\[
	0 <c < \gamma(x) <c^{-1} \qquad \textrm{for a.e. } x \in \Omega. \]
	Given $f \in H^{1/2}(\pa \Omega)$, let $u_f \in H^1(\Omega)$ denote the unique solution to the following Dirichlet boundary value problem:
	\begin{equation}\label{bvp}
		\left\{
		\begin{array}{rl}
			-\nabla \cdot (\gamma\nabla u_f)=  & 0  \textrm{ in } \Omega, \\
			u_f = & f \textrm{ on } \partial \Omega.
		\end{array} \right.
	\end{equation}
	The Dirichlet-to-Neumann map of $\gamma$, $\Lambda_\gamma$ is defined as the map that sends  
	\[
	f \in H^{1/2}(\pa \Omega) \mapsto \gamma\frac{\pa u_f}{\pa \nu}\bigg|_{\pa \Omega} \in H^{-1/2}(\pa \Omega),\]
	where  $\pa /\pa\nu$ is the outward pointing unit normal vector field on $\pa \Omega$. Here, $\gamma \frac{\pa u_f}{\pa \nu}$ is interpreted in the weak sense as follows: Given $g \in H^{1/2}(\pa \Omega)$, let $v_g \in H^1(\Omega)$ be any function such that $v_g|_{\pa \Omega} = g$. Then 
	\[
	\langle \Lambda_\gamma(f), g \rangle = \left\langle \gamma \frac{\pa u_f}{\pa \nu}, g \right\rangle := \int_\Omega \gamma \nabla u_f \cdot \nabla v_g\, dx. \]
	Physically, if $\gamma(x)$ represents the electrical conductivity at a point $x$ inside an object $\Omega$ and $f$ is the voltage applied on its boundary $\pa \Omega$, then the solution $u_f$ of \eqref{bvp} is precisely the induced electric potential inside $\Omega$. In this case, $\gamma \pa_\nu u_f |_{\pa \Omega}$ is the induced current flux density at the boundary and therefore, the map $\Lambda_\gamma$ encodes the set of all possible voltage and current measurements that can be made on the boundary.\\
	
	The inverse conductivity problem, first proposed by Alberto Calder{\'o}n in 1980 (\cite{Cal80}), asks whether we can determine the conductivity $\gamma$ from measurements on the boundary, encoded by $\Lambda_\gamma$. For there to be any hope of reconstruction, we first need the map $\gamma \mapsto \Lambda_\gamma$ to be injective. Calder{\'on} proved injectivity for a linearized version of the problem where $\gamma$ was assumed to be a small isotropic perturbation of the identity. For the full nonlinear problem, injectivity was first proved for $n \geq 3, \gamma \in C^2$  by Sylvester and Uhlmann in \cite{SU}. Their approach was to reduce the problem to a similar problem for the Schr{\"o}dinger equation at $0$ energy: let $q$ be a complex valued function in $\Omega$ such that $0$ is not a Dirichlet eigenvalue for $(-\Delta+q)$ on $\Omega$. Given $f \in H^{1/2}(\pa \Omega)$, let $u_f$ denote the unique solution to  the following boundary value problem:
	\begin{equation}\label{bvp-schrodinger}
		\left\{
		\begin{array}{rl}
			(-\Delta +q)u_f= & 0   \textrm{ in } \Omega \\
			u_f = & f \textrm{ on } \partial \Omega.
		\end{array} \right.
	\end{equation}
	The Dirichlet-to-Neumann map for $q$ is defined as the map $\Lambda_q : f \mapsto \frac{\partial u_f}{\partial \nu}\big|_{\pa \Omega}$. The corresponding inverse problem is to determine $q$ from $\Lambda_q$. Sylvester and Uhlmann showed that the inverse problem for the conductivity equation can be reduced to the inverse problem for the Schr{\"o}dinger equation with $q = \gamma^{-1/2}\Delta \gamma^{1/2}$. Next, the authors proved the injectivity of $q \mapsto \Lambda_q$ using the so-called Complex Geometrical Optics (CGO) solutions to $(-\Delta+q)u = 0$, defined globally in $\R^n$. These are solutions of the form $e^{x \cdot \zeta}(1+r_\zeta(x))$, where $\zeta \in \C^n$ is such that $\zeta \cdot \zeta = 0$ and $r_\zeta$ has certain decay properties as $|\zeta| \to \infty$. \\
	
	Once we know that $\gamma \mapsto \Lambda_\gamma$ is injective, we may try to find a constructive procedure for computing $\gamma$ from $\Lambda_\gamma$.  In \cite{nachmanrecon}, Nachman provided such a constructive procedure for computing $\gamma$ (resp., $q$) from $\Lambda_\gamma$ (resp., $\Lambda_q$) when $\gamma \in C^{1,1}$ (resp., $q \in L^\infty$). The procedure is based on the observation that CGO solutions satisfying certain decay conditions are uniquely determined by their restrictions to $\pa \Omega$. In turn, these restrictions can be characterized as the unique solutions of certain boundary integral equations on $\pa \Omega$. \\ 
	
	An interesting problem that has received considerable interest is  of finding the minimum regularity assumptions on $\gamma$ (or $q$) under which injectivity and the reconstruction procedure hold. This question is also of practical importance. For example, it was pointed out in \cite{Caro} that if $q$ arises from a Gaussian random field satisfying certain conditions, almost every instantiation of $q$ belongs to a Sobolev space of fixed negative order. For $n \geq 3$, the regularity assumption for uniqueness  was relaxed to $\gamma \in C^{3/2+}$ in \cite{brown96}, to $C^{3/2}$ in \cite{ppu03}, to $H^{3/2,2n+}$ in \cite{bro03}, and  to $\gamma \in C^1$ or $\gamma \in C^{0,1}$ with $\|\nabla \log \gamma\|_{L^{\infty}}$ small in \cite{HT13}. The smallness condition was removed in \cite{caro-rogers}. In \cite{ns14}, uniqueness was proved in 3 dimensions for conductivities in $H^{s,n/s}$ for $3/2<s<2$. It was also conjectured by Brown in \cite{bro03} that uniqueness holds for $\gamma \in H^{1,n}$ for all $n \geq 3$. This was proved for $n=3,4$ in \cite{haberman}.\\
	
	For the problem of reconstruction, Nachman's procedure in \cite{nachmanrecon} was adapted to the case of $\gamma \in C^1$ or $\gamma \in C^{0,1}$ with $\|\nabla \log \gamma\|_{L^\infty}$ sufficiently small in \cite{Garcia_2016}.  In this paper, we extend Nachman's reconstruction procedure to the Sobolev scale $H^{3/2,2n}$ for conductivities that are identically $1$ near $\pa \Omega$. Note that functions in $H^{3/2,2n}$ need not be $C^1$ or even Lipschitz, but do belong to the Zygmund space $C^1_*$ (c.f. \cite{Triebel})  of continuous functions $f$ such that
	\[
	\|f\|_{C^1_*} = \sup_{x \in \R^n} |f(x)| + \sup_{x, h \in \R^n, h \neq 0} \frac{|f(x+h)+f(x-h)-2f(x)|}{h} < \infty. \]
	To the best of our knowledge, this paper presents the first constructive result for the Calder{\'o}n problem in dimensions $n \geq 3$ for a class of conductivities that includes non-Lipschitz functions. We also note here that since the original version of this paper was first posted on arXiv, reconstruction results have also been proved for Lipschitz conductivities by Caro, Garc\'ia-Ferrero, and Rogers \cite{carogarcia}.\\
	
	\begin{table}[H]
	\centering
	\renewcommand{\arraystretch}{1.15}
	\begin{tabular}{@{}l l l@{}}
		\toprule
		\textbf{Regularity of $\gamma$} & \textbf{Nonconstructive} & \textbf{Constructive} \\
		\midrule\midrule
		$C^{3/2+}$ & \cite{brown96} & $\downarrow$ \\
		$C^{3/2}$ & \cite{ppu03} & $\downarrow$ \\
		$C^{1}$ & \cite{HT13} & \cite{Garcia_2016} \\
		$C^{0,1}$ with small $\|\nabla\log\gamma\|_{L^\infty}$ & \cite{HT13} & \cite{Garcia_2016} \\
		$C^{0,1}$ & \cite{caro-rogers} & \cite{carogarcia} \\
		$H^{s, n/s}$ with $3/2<s<2$ for $n=3$ & \cite{ns14} & open \\
		$H^{3/2,\,2n+}$ & \cite{bro03} &  \cite{Garcia_2016} (since $H^{3/2,2n+}\subset C^{1+}$) \\
		$H^{3/2,\,2n}$ with $\gamma \equiv 1$ near $\partial\Omega$ & $\rightarrow$  & \textbf{this work} \\
		$H^{3/2,2n}$ (without boundary condition) & $\downarrow$ & open \\
		$H^{1,n}$ for $n=3,4$ & \cite{haberman} & open \\
		$H^{1,n}$ for $n\ge 5$ & open & open \\
		\bottomrule
	\end{tabular}
	\caption{Existing literature on the uniqueness of conductivity $\gamma$.}
	\label{tab:uniqueness_gamma_referee}
	\end{table}
	
	We assume that $\gamma \equiv 1$ near $\partial \Omega$ to ensure that extending $\gamma$ to $\R^n$ by setting it equal to $1$ on $\R^n \setminus \Omega$ maintains $H^{3/2,2n}$ regularity. While the reduction to the Schr{\"o}dinger equation for $\gamma \in H^{3/2,2n}$ produces $q \in H^{-1/2,2n}$ (which was treated in \cite{bro03} in the context of uniqueness), the validity of this reduction for the purposes of reconstruction, in the sense of equality of the corresponding Dirichlet-to-Neumann maps (see Proposition \ref{reduction}) had not been established before at this level of regularity. \\
	
	It is natural to ask whether the condition that $\gamma \equiv 1$ near $\partial \Omega$ can be dropped. This was done in the case of $C^1$ conductivities in \cite{Garcia_2016}, where the authors first recover $\gamma|_{\pa \Omega}$ and $\nabla \gamma|_{\partial \Omega}$ from $\Lambda_\gamma$ and use it to extend $\gamma$ to a larger domain while maintaining $C^1$ regularity and ensuring that the extension is $\equiv 1$ near the boundary of the larger domain. This allows them to reduce the general case to the case where $\gamma \equiv 1$ near $\pa \Omega$. However, since functions in $H^{3/2,2n}(\Omega)$ are not necessarily Lipschitz, their arguments no longer apply, and such a reduction would need additional ideas not pursued in this paper. \\
	
	Another question of interest is of stability of the map $\gamma \mapsto \Lambda_\gamma$. It was shown by Alessandrini in \cite{alessandrini88} that under the a priori assumption
	\[
	\|\gamma_j\|_{H^s(\Omega)} \leq M, \qquad s > n/2+2, \, j=1,2, \]
	we have a stability estimate of the form
	\[
	\|\gamma_1-\gamma_2\|_{L^{\infty}(\Omega)} \leq C\left\{ \left|\log\|\Lambda_{\gamma_1}-\Lambda_{\gamma_2}\|_{H^{1/2} \to H^{-1/2}}\right|^{-\sigma} +\|\Lambda_{\gamma_1}-\Lambda_{\gamma_2}\|_{H^{1/2} \to H^{-1/2}} \right\} \]
	where $\sigma = \sigma(n,s) \in (0,1)$. Subsequently, the a priori assumptions were relaxed to $\|\gamma_j\|_{H^{2,\infty}} \leq M$ in \cite{alessandrini90,alessandrini91}. Such logarithmic estimates were shown to be optimal up to the value of the exponent by Mandache in \cite{Mandache} via explicit examples. Later, stability was proved for conductivities bounded a priori in $C^{1,\frac{1}{2}+\epsilon}\cap H^{n/2+\epsilon}$ with $\pa \Omega$ smooth by Heck in \cite{Heck-Stability} and for a priori bounds in $C^{1,\epsilon}(\overline{\Omega})$ with $\pa \Omega$ Lipschitz by Caro, Garc{\' i}a and Reyes in \cite{Caro-Stability}. In this paper, we prove a similar log-type stability estimate with $\|\gamma_j\|_{H^{2-s, n/s}(\Omega)} \leq M$ for some $0<s<1/2$. In the H{\"o}lder scale, this corresponds to $C^{1,\epsilon}$ as in 
	\cite{Caro-Stability}, with $\epsilon = 1-2s$. However, our use of the Sobolev scale allows us to avoid Bourgain-type spaces, leading to a much simpler and shorter proof.\\
	
	\begin{table}[H]
		\centering
		\renewcommand{\arraystretch}{1.15}
	\begin{tabular}{ll}
		\toprule
		\textbf{A priori assumption on $\gamma$} &  \textbf{Reference}\\
		\midrule\midrule
		$H^{s}(\Omega)$ with $s > \frac{n}{2}+2$ & \cite{alessandrini88} \\
		$H^{2,\infty}$ & \cite{alessandrini90,alessandrini91} \\
		$C^{1,\frac12+\varepsilon}\cap H^{\frac n2+\varepsilon}$ with smooth $\partial\Omega$ & \cite{Heck-Stability} \\
		$C^{1,\varepsilon}(\Omega)$ for some $\varepsilon>0$ &\cite{Caro-Stability}\\
		$H^{2-s,\frac ns}$ for some $0<s<\frac12$ with $\gamma=1$ near $\partial\Omega$ &
		\multirow{2}{*}{\textbf{this work}} \\
		(Note: $H^{2-s,\,n/s}\subset C^{1,\,1-2s}$ for $0<s<1/2$.) & \\
		\bottomrule
	\end{tabular}
	\caption{Existing literature on logarithmic-type stability estimates.}
	\label{tab:stability_gamma}
	\end{table}
	
	The main results of this paper are summarized in the following theorem: 
	\begin{theorem}\label{mainth}
		Let $\Omega$ be a  bounded Lipschitz domain in $\R^n$, $n \geq 3$, and $D\subset \Omega$ a compact subset. Assume $\gamma \in H^{3/2,2n}(\Omega)$ is a positive real valued function satisfying
		\begin{equation}\label{ellip}
			0 < c < \gamma(x) <c^{-1} \qquad \mbox{for a.e. } x \in \Omega 
		\end{equation}
		and $\gamma \equiv 1$ in $\Omega \setminus D$. Then,
		\begin{enumerate}[(a)]
			
			\item One can determine $\gamma$ from the knowledge of the map $\Lambda_\gamma :H^{1/2}(\pa \Omega) \to H^{-1/2}(\pa \Omega)$ in a constructive way. Moreover, 
			
			\item We have the following stability estimate: Let $\gamma_j \in H^{3/2,2n}(\Omega)$, $j=1,2$,  be such that $\gamma_j \equiv 1$ in $\Omega \setminus D$ and satisfy the ellipticity bound \eqref{ellip}. Suppose in addition that  $\|\gamma_j\|_{H^{2-s,n/s}(\Omega)} \leq M$  for some $0<s<1/2$, and let $0\leq \alpha <1$. Then there exist $C = C\left(\Omega,n,c, M,s,\alpha\right)>0$ and $0 < \sigma = \sigma(n,s,\alpha) <1$ such that
			\begin{equation}\label{gamma-stab}
				\|\gamma_1 - \gamma_2\|_{C^{\alpha}(\overline{\Omega})} \leq C\left( |\log \|\Lambda_{\gamma_1}-\Lambda_{\gamma_2}\|_{H^{1/2}\to H^{-1/2}} |^{-\sigma} +\|\Lambda_{\gamma_1}-\Lambda_{\gamma_2}\|_{H^{1/2}\to H^{-1/2}} \right). 
			\end{equation}
		\end{enumerate}
	\end{theorem}
	
	Note that the constant C in \eqref{gamma-stab} is independent of $D$. As usual, this result will be obtained as a consequence of the corresponding result for the Schr{\"o}dinger equation:
	\begin{theorem}\label{mainth-schrod}
		Let $\Omega$ be a bounded Lipschitz domain in $\R^n$, $n \geq 3$, and $D\subset \Omega$ a compact subset. Let $q \in H^{-1/2,2n}(\Omega)$ be such that $\supp(q) \subset D$, and assume that $0$ is not a Dirichlet eigenvalue of the  boundary value problem \eqref{bvp-schrodinger}. Then,
		\begin{enumerate}[(a)]
			\item One can determine $q$ from the knowledge of the map $\Lambda_q : H^{1/2}(\pa \Omega) \to H^{-1/2}(\pa \Omega)$ in a constructive way. Moreover,
			\item We have the following stability estimate: Let $q_j \in H^{-1/2,2n}(\Omega), \, j=1,2$ be such that $\supp(q_j) \subset D$, and assume that $0$ is not a Dirichlet eigenvalue of \eqref{bvp-schrodinger} for $q = q_1, q_2$.  Suppose in addition that $\|q_j\|_{H^{-s,n/s}} \leq M$ for some $0<s<1/2$. Then there exist $C = C\left(\Omega,n, M,s\right)>0$ and $0 < \sigma = \sigma(n,s) <1$ such that
			\begin{equation}\label{q-stab}
				\|q_1-q_2\|_{H^{-1}} \leq C \left( |\log \|\Lambda_{q_1}-\Lambda_{q_2}\|_{H^{1/2} \to H^{-1/2}} |^{-\sigma} +\|\Lambda_{q_1}-\Lambda_{q_2}\|_{H^{1/2}\to H^{-1/2}} \right).
			\end{equation} 
		\end{enumerate}
	\end{theorem}
	
	\vspace{5mm}
	While this paper deals only with the full data Calder{\' o}n problem,  we note here that the problem of partial data, where measurements are made on only a part of the boundary is also of significant interest. Several results have been obtained on uniqueness (\cite{bukhgeim-uhlmann, KSjU, Isakov-partial, Kenig-Salo-I}), minimum regularity (\cite{Knudsen-partial, Zhang-partial, Rodriguez-partial, Krupchyk-partial}), reconstruction (\cite{nachman-partial, ammari-partial, yernat-partial}) and stability (\cite{heck-partial, lai-partial, caro-partial}). We refer the reader to \cite{Kenig-Salo-II} for a survey on the Calder{\'o}n problem with partial data. The problem for $n=2$ is also by now well understood. Uniqueness was first proved for $C^2$ conductivities in \cite{Nachman}. The regularity assumptions were later relaxed to $H^{1,2+}$ in \cite{BrownUhlmann}, to $L^\infty$ in \cite{Paivarinta} and to $L^{2+}$ in \cite{bukhgeim}. Nachman's reconstruction procedure has also been extended to $L^\infty$ conductivities in the plane that are  $ 1$ near the boundary in \cite{recon-2d}. Stability estimates (\cite{Blaasten}) and various partial data results (\cite{Imanuvilov1, Imanuvilov2, Imanuvilov3}) are also known. \\
	
	Here is a short outline of the paper: In Section 2 we collect some function-space preliminaries needed for the low-regularity analysis, including Bessel potential spaces and a Kato–Ponce type product estimate in those spaces,  the weighted $L^2$ spaces and $k$-scaled Sobolev norms used in the construction of CGO solutions, and related results. In Section 3 we establish well-posedness for the Schr{\"o}dinger equation for potentials in $H^{-1/2,2n}$ supported in $D$, and justify the reduction from the conductivity equation to the Schrödinger equation at the required regularity. In Section 4 we construct CGO solutions in $\R^n$. Section 5 uses these solutions to prove uniqueness and the validity of Nachman's reconstruction procedure. Finally, Section 6 proves the logarithmic-type stability estimates.
	
\section{Function Space Preliminaries}
	
    We begin by recalling the class of Bessel potential spaces $H^{s,p}(\R^n)$, defined by the norms
	\[
	\|f\|_{H^{s,p}} = \|(I-\Delta)^{s/2}f\|_{L^p}, \qquad s \in \R, \, p \geq 1. \]
	For a bounded Lipschitz domain $\Omega \subset \R^n$, $H^{s,p}(\Omega)$ is defined as the space of $H^{s,p}(\R^n)$ functions restricted to $\Omega$, i.e.,
	\[
	H^{s,p}(\Omega) := \{ u|_{\Omega} : u \in H^{s,p}(\R^n)\} \]
	and is equipped with the quotient norm
	\[
	\|f\|_{H^{s,p}(\Omega)} = \inf \{ \|u\|_{H^{s,p}(\R^n) }: u|_{\Omega} = f \}. \]
	When $p=2$, $H^{s,p}(\R^n)$ and $H^{s,p}(\Omega)$ reduce to the standard Sobolev spaces $H^s(\R^n)$ and $H^s(\Omega)$ respectively. For a fixed compact set $D\subset \Omega$, we also define
	\[
	H^{s,p}_D := \{u \in H^{s,p}(\R^n): \supp(u) \subset D\}.\]
	With this notation, our assumptions on the conductivity function $\gamma$ can be stated as $\gamma -1 \in H^{3/2,2n}_D$. We will show later that the corresponding potential $q = \gamma^{-1/2}\Delta \gamma^{1/2} \in H^{-1/2, 2n}_D$, and establish bounds for the multiplication operator $m_q: f \mapsto qf$.\\
	
	\begin{samepage}
		\noindent \emph{Remarks on Notation.}
		\begin{enumerate}[(i)]
			\item Henceforth, we will use $qu$ and $m_q(u)$ interchangeably. 
			\item We will use the notation $A \lsim B$ to mean that there exists $C>0$ independent of variables appearing in $A$ and $B$ such that $A \leq CB$.  To indicate that $C$ depends on quantities $D,E, \ldots$, we write $A\lsim_{D,E,\ldots} B$.
			\item We also write $A \asymp B$ if $A\lsim B$ and $B\lsim A$.
		\end{enumerate}
	\end{samepage}
	
	\medskip
	
	In order to establish the required bounds on $m_q$, we will need the following important lemma.
	
	\begin{lemma}\label{katoponce}
		Let $s >0$ and $p \in (2,\infty)$ be such that $p \geq n/s$. Then for all $f,g \in \mathcal{S}(\R^n)$, we have
		\[
		\|fg\|_{H^{s,p'}} \lsim \|f\|_{H^s}\|g\|_{H^s}, \]
		where $1/p+1/p' = 1$.
	\end{lemma}
	\begin{proof}
		The Kato-Ponce inequality (ref. \cite{katoponce-grafakos, katoponce-orig}) implies that
		\[
		\|fg\|_{H^{s,p'}} \lsim \|f\|_{H^s}\|g\|_{L^r}+\|f\|_{L^r}\|g\|_{H^s} \]
		where $1/r = 1/p'-1/2 = 1/2-1/p$. Now applying the Sobolev embedding theorem, we get 
		\[
		\|fg\|_{H^{s,p'}} \lsim \|f\|_{H^s}\|g\|_{H^t}+\|f\|_{H^t}\|g\|_{H^s}, \]
		where $1/2-t/n = 1/r$, or equivalently, $t = n/2-n/r =n/p$. Now the estimate follows from the fact that $t \leq s$.
	\end{proof}
	
	Next, we introduce certain weighted $L^2$ and Sobolev spaces that are used in the construction of CGO solutions.
	
	\begin{defn}
		Let $\delta \in \R$. We define the weighted $L^2$  space $L^2_\delta(\R^n)$ by the norm
		\[
		\|u\|_{L^2_\delta} = \left( \int_{\R^n} (1+|x|^2)^{\delta}|u(x)|^2\, dx \right)^{1/2}. \]
		For any $m \in \N$, we define the corresponding weighted Sobolev spaces $H^m_\delta(\R^n)$ through the norms
		\[
		\| u\|_{H^m_\delta} = \sum_{|\alpha|\leq m} \|\partial^\alpha u\|_{L^2_\delta}. \]
		Finally, notice that $L^2_\delta$ and $L^2_{-\delta}$ are duals of each other with respect to $\langle \cdot, \cdot \rangle_{L^2}$. Motivated by this, we define the negative order spaces $H^{-m}_\delta(\R^n)$ for $m \in \N $ as duals of $H^m_{-\delta}(\R^n)$.
	\end{defn}
	
	We will also need the following scaled versions of these Sobolev norms.
	
	\begin{defn}
		Let $ s \in \R, k \geq 1$. We define $H^s[k](\R^n)$ through the norms
		\[
		\|u\|_{H^s[k]} = \|(k^2-\Delta)^{s/2}u\|_{L^2} =  \frac{1}{(2\pi)^{n/2}}\left( \int (k^2+|\xi|^2)^s |\widehat{u}(\xi)|^2 \, d\xi \right)^{1/2}. \]
		Note that $H^s[k](\R^n)$ and $H^{-s}[k](\R^n)$ are dual to each other with respect to $\langle \cdot , \cdot \rangle_{L^2}$. If $s \in \N$, then for $\delta \in \R$, we define $H^s_{\delta}[k](\R^n)$ through the norms
		\[
		\|u\|_{H^s_{\delta}[k]} = \sum_{|\alpha|\leq s} k^{s-|\alpha|}\|\partial^{\alpha}u\|_{L^2_\delta}. \]
		Finally, for negative integers $s$, we define $H^s_{\delta}[k](\R^n)$ as the dual of $H^{-s}_{-\delta}[k](\R^n)$.
	\end{defn}
	Just as in the case of the usual negative order Sobolev spaces, we have the following characterization of $H^{-m}_{\delta}[k], m \in \N$. The proof is similar to the usual $H^{-m}$ case and therefore is omitted.
	\begin{prop}\label{altern}
		Let $m \in \N, \delta \in \R, k \geq 1$. For every $u \in H^{-m}_{\delta}[k](\R^n)$, there exist $\{u_\alpha \in L^2_\delta(\R^n):|\alpha|\leq m\}$ such that
		\[
		\langle u, v \rangle = \sum_{|\alpha|\leq m} \langle \partial^{\alpha} u_{\alpha}, v \rangle = \sum_{|\alpha|\leq m} (-1)^{|\alpha|}\langle u_{\alpha}, \partial^{\alpha}v \rangle_{L^2} \quad \forall v \in H^{m}_{-\delta}[k](\R^n). \]
		Moreover,  $u_\alpha$ can be chosen to satisfy
		\[
		\sum_{|\alpha|\leq m}k^{-(m-|\alpha|)} \|u_{\alpha}\|_{L^2_\delta} = \|u\|_{H^{-m}_{\delta}[k]}. \]
	\end{prop}
	
	We record the following simple inequality for future use.
	\begin{lemma}\label{lcompact}
		Let $m \in \Z, k \geq 1$ and $\delta,\eta \in \R$. Fix $\varphi \in C_c^{\infty}(\R^n)$. Then 
		\[
		\|\varphi u \|_{H^{m}_{\delta}[k]} \lsim_{\varphi,m,\delta,\eta} \|u\|_{H^{m}_{\eta}[k]}, \qquad u \in H^{m}_{\eta}[k](\R^n). \]
	\end{lemma}
	\begin{proof}
		For $m \geq 0$, this follows from the fact that for any multi-index $\alpha$,  $\pa^{\alpha}\varphi(x) (1+|x|^2)^{(\delta-\eta)/2}$ is bounded above. Now suppose $m <0$. Let $v \in H^{-m}_{-\delta}[k](\R^n)$. Then
		\begin{eqnarray*}
			|\langle \varphi u, v\rangle_{L^2}| &= & |\langle u, \varphi v\rangle | \\
			&\leq & \|u\|_{H^{m}_{\eta}[k]}\|\varphi v\|_{H^{-m}_{-\eta}[k]} \\
			&\leq & \|u\|_{H^{m}_{\eta}[k]}\|v\|_{H^{-m}_{-\delta}[k]}.
		\end{eqnarray*}
		Taking the supremum of the left hand side over all $v$ with $\|v\|_{H^{-m}_{-\delta}[k]} \leq 1$ gives us the desired result.
	\end{proof}
	
	We now establish $H^s[k]$ bounds on the multiplication operator $m_q:f \mapsto qf$ when $q$ belongs to a negative order Bessel potential space. We closely follow the proof of  Proposition 3.2 in \cite{Caro}. 
	
	\begin{prop}\label{mV}
		Let $s >0$ and $p \in (2,\infty)$ be such that $p \geq n/s$.  Suppose $V \in H^{-s,p}(\R^n)$. Then for $k \geq1$,
		\begin{equation}\label{sbound}
			\|Vf\|_{H^{-s}[k]} \lsim \omega(k)\|f\|_{H^{s}} \lsim \omega(k)\|f\|_{H^s[k]}, \qquad \forall f \in H^{s}(\R^n),  
		\end{equation}
		where $\omega$ is a positive function on $[1,\infty)$ such that $\omega(k) \to 0$ as $k \to\infty$. If in addition we have $0<s\leq 1$, then
		\begin{eqnarray}
			\|Vf\|_{H^{-1}[k]} &\lsim & k^{-(1-s)}\omega(k)\|f\|_{H^1}, \quad \textrm{and} \label{1kbound}\\
			\|Vf\|_{H^{-1}[k]} &\lsim & k^{-2(1-s)}\omega(k)\|f\|_{H^1[k]}. \label{1bound}
		\end{eqnarray}
		Note that the implicit constants (indicated by $\lsim$) in the right hand sides of  \eqref{sbound},\eqref{1kbound}, and \eqref{1bound} are independent of $k$.
	\end{prop}
	\begin{proof}
		By duality, it suffices to prove that
		\[
		|\bra Vf, g\ket| = |\bra V, fg\ket| \lsim \omega(k)\|f\|_{H^{s}}\|g\|_{H^{s}[k]} \qquad \textrm{for all } f,g \in \mathcal{S}(\R^n),\]
		where $\langle \cdot, \cdot \rangle$ denotes the duality pairing in $\mathcal{S}'(\R^n)\oplus \mathcal{S}(\R^n)$. Let $W = (I-\Delta)^{-s/2}V \in L^p(\R^n)$. Then we have,
		\[
		\langle V, fg\rangle = \bra (I-\Delta)^{s/2}W, fg\ket = \bra W, (I-\Delta)^{s/2}(fg)\ket. \]
		Let $\varphi \in C_c^\infty(\R^n;[0,1])$ be such that $\int_{\R^n}\varphi(x)dx = 1$. We consider the sequence of mollifiers $\varphi_\eps(x) := \eps^{-n}\varphi(x/\eps)$ and define $W_\eps := \varphi_\eps * W$. Choosing $t \in (s-n/p,s)$, we may write
		\begin{eqnarray*}
			\bra V, fg \ket &=& \bra W_\eps, (I-\Delta)^{s/2}(fg) \ket + \bra W-W_\eps, (I-\Delta)^{s/2}(fg)\ket\\
			&=& \bra (I-\Delta)^{t/2}W_\eps, (I-\Delta)^{(s-t)/2}(fg)\ket +\bra W-W_\eps, (I-\Delta)^{s/2}(fg)\ket.
		\end{eqnarray*}
		Now, by the Riesz representation theorem,
		\begin{equation}\label{holder}
			|\bra V, fg\ket| \leq \|(I-\Delta)^{t/2}W_\eps\|_{L^q}\|(I-\Delta)^{(s-t)/2}(fg)\|_{L^{q'}}+\|W-W_\eps\|_{L^p}\|(I-\Delta)^{s/2}(fg)\|_{L^{p'}},
		\end{equation}
		where $q=n/(s-t)$ and $p',q'$ are conjugate exponents of $p,q$ respectively. Since $t > s-n/p$, we have $q >p$ and therefore by Young's convolution inequality,
		\begin{eqnarray*}
			\|(I-\Delta)^{t/2}W_\epsilon\|_{L^q} & = & \|((I-\Delta)^{t/2}\varphi_\eps)*W\|_{L^q} \\
			&\leq & \|(I-\Delta)^{t/2}\varphi_\eps\|_{L^r}\|W\|_{L^p}\\
			&=& \|\varphi_\eps\|_{H^{r,t}}\|W\|_{L^p},
		\end{eqnarray*}
		where $1/p+1/r = 1+1/q$. Now using the scaling property $\|\varphi_\eps\|_{H^{r,t}} \asymp \eps^{-t-n+n/r}$ (ref. \cite{Triebel}, Corollary 5.16), we get
		\begin{equation*}
			\|(I-\Delta)^{t/2}W_\eps\|_{L^q}  \lsim  \eps^{-t+n/q-n/p}\|W\|_{L^p}.
		\end{equation*}
		Also, by Lemma \ref{katoponce},
		\begin{eqnarray*}
			\|(I-\Delta)^{(s-t)/2}(fg)\|_{L^{q'}} & \lsim & \|f\|_{H^{s-t}}\|g\|_{H^{s-t}}, \quad \textrm{and} \\
			\|(I-\Delta)^{s/2}(fg)\|_{L^{p'}} &\lsim & \|f\|_{H^s}\|g\|_{H^s}.
		\end{eqnarray*}
		Therefore, from \eqref{holder}, we get
		\begin{eqnarray*}
			|\bra V, fg\ket | &\lsim & \eps^{-t+n/q-n/p}\|W\|_{L^p}\|f\|_{H^{s-t}}\|g\|_{H^{s-t}} + \|W-W_\eps\|_{L^p}\|f\|_{H^s}\|g\|_{H^s} \\
			&\lsim & \eps^{-t+n/q-n/p}\|W\|_{L^p}\|f\|_{H^{s-t}}\|g\|_{H^{s-t}[k]} + \|W-W_\eps\|_{L^p}\|f\|_{H^s}\|g\|_{H^s[k]}\\
			&\lsim & (\eps^{-t+n/q-n/p}k^{-t}\|W\|_{L^p}+\|W-W_\eps\|_{L^p})\|f\|_{H^s}\|g\|_{H^s[k]}.
		\end{eqnarray*}
		Here we have used the easy estimate $\|h\|_{H^{s-t}[k]} \lsim k^{-t}\|h\|_{H^s[k]}$ for any $h \in \mathcal{S}(\R^n)$. Note that $n/q-n/p \geq (s-t)-s  = -t$. Now choose 
		$$\eps = k^{-1/4}$$ 
		and set
		\begin{equation}\label{omega}
			\omega(k) = k^{-t/2}\|W\|_{L^p}+\|W-W_{k^{-1/4}}\|_{L^p}.
		\end{equation}
		Then we get
		\begin{equation}\label{bilinears}
			|\bra V, fg\ket| \lsim \omega(k)\|f\|_{H^s}\|g\|_{H^s[k]} \lsim \omega(k)\|f\|_{H^s[k]}\|g\|_{H^s[k]} 
		\end{equation}
		where $\omega(k) \to 0$ as $k \to \infty$. This proves \eqref{sbound}. Now \eqref{1kbound} and \eqref{1bound} follow from the fact that if $0<s\leq 1$,
		\begin{eqnarray}
			|\bra V, fg\ket| &\lsim & \omega(k)\|f\|_{H^s}\|g\|_{H^s[k]} \lsim \omega(k)k^{-(1-s)}\|f\|_{H^1}\|g\|_{H^1[k]}, \label{bilinear1k}\\
			|\bra V,fg\ket| &\lsim & \omega(k)\|f\|_{H^s[k]}\|g\|_{H^s[k]} \lsim \omega(k)k^{-2(1-s)}\|f\|_{H^1[k]}\|g\|_{H^1[k]}. \label{bilinear1} 
		\end{eqnarray}
	\end{proof}
	
	If in addition, $V$ is compactly supported, the multiplication operator $m_V$ can be extended to $H^s_{\delta}[k]$ spaces.
	\begin{corollary}\label{multiplierbound}
		Let $0<s<1$ and $q \in H^{-s,n/s}(\R^n)$ be such that $\supp(q)$ is compact. Suppose $\delta, \eta \in \R$. Then $m_q: f \mapsto qf$  satisfies the norm bounds
		\begin{eqnarray}
			\|m_q f\|_{H^{-1}_{\delta}[k]} &\lsim & k^{-(1-s)}\omega(k)\|f\|_{H^1_{\eta}}, \label{bilinear1kdelta}\\
			\|m_q f\|_{H^{-1}_{\delta}[k]} &\lsim & k^{-2(1-s)}\omega(k)\|f\|_{H^1_{\eta}[k]} \label{bilinear1delta}.
		\end{eqnarray}
		where $\omega$ is the positive function on $[1,\infty)$ defined by \eqref{omega}. In particular, it satisfies $\omega(k) \to 0$ as $k \to \infty$. 
	\end{corollary}
	\begin{proof}
		Let $\varphi \in C_c^{\infty}(\R^n)$ be such that $\varphi \equiv 1$ on $\supp(q)$. Then by \eqref{bilinear1k}, for all $f,g \in \mathcal{S}(\R^n)$,
		\begin{eqnarray*}
			|\langle qf,g\rangle_{L^2}| =|\langle q, fg \rangle| &=& |\langle q, (\varphi f)(\varphi g)\rangle | \\
			&\lsim & \omega(k)k^{-(1-s)}\|\varphi f\|_{H^1}\|\varphi g\|_{H^1[k]}\\
			&\lsim & \omega(k)k^{-(1-s)}\|f\|_{H^1_\eta}\|g\|_{H^1_{-\delta}[k]} \quad \textrm{by Lemma \ref{lcompact}}. 
		\end{eqnarray*}
		Now \eqref{bilinear1kdelta} follows by density and duality. \eqref{bilinear1delta} similarly follows from \eqref{bilinear1}.
	\end{proof}
	
	\section{Well-posedness and reduction to the Schr{\"o}dinger equation for rough coefficients}
	
	We begin this section by establishing the well-posedness of the Dirichlet boundary value problem \eqref{bvp-schrodinger} when $q \in H^{-1/2,2n}_D$. 
	
	\begin{prop}\label{fredholm}
		Let $\Omega \subset \R^n$ be a bounded Lipschitz domain, $D\subset \Omega$ be a compact subset, and $q \in H^{-1/2,2n}_D$.
		\begin{enumerate}[(a)] 
			
			\item The multiplication operator $m_q: C^{\infty}(\Omega) \to \mathcal{D}'(\Omega)$ defined by 
			$$\langle m_q(\varphi),\psi\rangle =\langle q, \varphi \psi \rangle, \qquad \varphi \in C^\infty(\Omega), \ \psi \in C_c^\infty(\Omega),$$
			extends to a continuous map $H^1(\Omega) \to H^{-1}(\Omega)$, and is compact.
			
			\item (The Fredholm Alternative) Exactly one of the following must be true:
			\begin{enumerate}[(i)]
				\item For any $f \in H^{1/2}(\pa \Omega)$ and $F \in H^{-1}(\Omega)$, there exists a unique $u \in H^1(\Omega)$ such that
				\begin{equation*}
					\left\{
					\begin{array}{rl}
						(-\Delta +m_q)u= & F   \textrm{ in } \Omega \\
						u = & f \textrm{ on } \partial \Omega.
					\end{array} \right.
				\end{equation*}
				Moreover, there exists $C=C(q,\Omega)>0$ such that
				\[
				\|u\|_{H^1(\Omega)} \leq C(\|f\|_{H^{1/2}(\pa \Omega)} + \|F\|_{H^{-1}(\Omega)}). \]
				
				\item There exists $u \in H^1(\Omega),  u \not\equiv 0$, such that
				\begin{equation*}
					\left\{
					\begin{array}{rl}
						(-\Delta +m_q)u= & 0   \textrm{ in } \Omega \\
						u = & 0\textrm{ on } \partial \Omega.
					\end{array} \right.
				\end{equation*}
				That is, $0$ is a Dirichlet eigenvalue of $(-\Delta+m_q)$ on $\Omega$.
			\end{enumerate}
		\end{enumerate}
	\end{prop}
	\begin{proof}
		It follows from Proposition \ref{mV} that $m_q$ maps $H^1(\Omega) \to H^{-1/2,2}_D$. The compactness of $m_q: H^1(\Omega) \to H^{-1}(\Omega)$ follows from the compactness of the inclusion $H^{-1/2,2}_D \hookrightarrow H^{-1,2}_D  \hookrightarrow H^{-1}(\Omega)$ (ref. \cite{mclean}, Theorem 3.27).\\
		
		Next, we note that $(-\Delta + m_q): H^1_0(\Omega) \to H^{-1}(\Omega)$ is Fredholm, since $-\Delta: H^1_0(\Omega) \to H^{-1}(\Omega)$ is invertible and $m_q$ is compact. Therefore, (b) follows from standard Fredholm theory.
	\end{proof}
	
	As usual, if $0$ is not a Dirichlet eigenvalue of $(-\Delta+q)$, we define the Dirichlet-to-Neumann map $\Lambda_q: H^{1/2}(\pa \Omega) \to H^{-1/2}(\pa \Omega)$ by duality: Given $f \in H^{1/2}(\pa \Omega)$, let $u \in H^1(\Omega)$ be the unique solution of \eqref{bvp-schrodinger}. Then
	\[
	\langle \Lambda_q f, g\rangle = \int_{\pa \Omega} \Lambda_q (f)g \, d\sigma = \int_\Omega \nabla u \cdot \nabla v\, dx +\langle m_q u, v\rangle_{L^2(\Omega)}, \qquad g \in H^{1/2}(\pa \Omega) \]
	where $d\sigma$ is the surface measure on $\pa \Omega$ and $v \in H^1(\Omega)$ is any function such that $v|_{\pa \Omega} =g$. We also get the following integral identity as a consequence of the symmetry of the multiplication operator $m_q$:
	\begin{prop}\label{integral}
		Let $q_1,q_2 \in H^{-1/2,2n}_{D}(\Omega)$ be such that $0$ is not a Dirichlet eigenvalue of $(-\Delta +m_{q_j})$ on $\Omega$, $j=1,2$. Let $u_1,u_2 \in H^1(\Omega)$ be solutions of $(-\Delta+m_{q_j})u_j = 0$, $j=1,2$. Then
		\begin{equation}\label{int-identity}
			\int_{\pa \Omega}(\Lambda_{q_1}-\Lambda_{q_2})u_1\cdot u_2\, d\sigma = \int_{\Omega} (m_{q_1}-m_{q_2})u_1 \cdot u_2 \, dx 
		\end{equation}
		where $d\sigma$ is the surface measure on $\pa \Omega$.
	\end{prop}
	
	We will now show that when $\gamma \in H^{3/2,2n}(\Omega)$ and is such that $\gamma \equiv 1$ in $\Omega \setminus D$, the boundary value problem \eqref{bvp} reduces to the corresponding problem \eqref{bvp-schrodinger} for the Schr{\"o}dinger equation.
	
	\begin{prop}\label{reduction}
		Let $\gamma \in H^{3/2,2n}( \Omega)$ be such that
		\[
		0 < c < \gamma(x) < c^{-1} \qquad \textrm{a.e. on }\Omega \]
		and $\gamma \equiv 1$ in $\Omega \setminus D$. Extend $\gamma$ to all of $\R^n$ by defining $\gamma \equiv 1$ on $\R^n \setminus \Omega$ and define $q = \Delta \sqrt{\gamma}/\sqrt{\gamma}$. 
		\begin{enumerate}[(a)]
			
			\item $q \in H^{-1/2,2n}_{D}( \Omega)$.
			
			\item $u \in H^1(\Omega)$ solves
			\begin{equation}\label{bvp1}
				\left\{
				\begin{array}{rl}
					-\nabla \cdot (\gamma\nabla u)=  & 0  \textrm{ in } \Omega \\
					u = & f \in H^{1/2}(\pa \Omega)
				\end{array} \right.
			\end{equation}
			if and only if $w = \gamma^{1/2}u \in H^1(\Omega)$ solves
			\begin{equation}\label{bvp2}
				\left\{
				\begin{array}{rl}
					(-\Delta +q)w= & 0   \textrm{ in } \Omega \\
					w = & f \textrm{ on } \partial \Omega.
				\end{array} \right.
			\end{equation}
			
			\item[(c)] $0$ is not a Dirichlet eigenvalue of $(-\Delta+m_q)$ on $\Omega$ and $\Lambda_q = \Lambda_\gamma$.
		\end{enumerate}
	\end{prop}
	\begin{proof}
		\begin{enumerate}[(a)]
			
			\item That $\supp(q) \subset D$ follows from the fact that $\gamma \equiv 1$ outside $D$. Next consider the identity
			\begin{eqnarray*}
				\frac{\Delta \sqrt{\gamma}}{\sqrt{\gamma}} &=& \frac{1}{2}\Delta \log \gamma + \frac{1}{4}|\nabla \log \gamma|^2 \\
				\Rightarrow \|q\|_{H^{-1/2,2n}} &\lsim & \|\Delta \log \gamma\|_{H^{-1/2,2n}} + \||\nabla \log \gamma|^2\|_{H^{-1/2,2n}}\\
				&\lsim & \|\log \gamma\|_{H^{3/2,2n}} + \||\nabla \log \gamma|^2\|_{L^{n}}  \quad (\textrm{as }L^n(\R^n) \hookrightarrow H^{-1/2,2n}(\R^n) ) \\
				&= &  \|\log \gamma\|_{H^{3/2,2n}}+\|\nabla \log \gamma\|_{L^{2n}}^2 \\
				&\lsim &  \|\log \gamma\|_{H^{3/2,2n}} + \|\log \gamma\|_{H^{1,2n}}^2 \\
				&\lsim & \|\log \gamma \|_{H^{3/2,2n}}+\|\log \gamma \|_{H^{3/2,2n}}^2 \quad (\textrm{as } H^{3/2,2n}(\R^n) \hookrightarrow H^{1,2n}(\R^n))
			\end{eqnarray*}
			by the monotonicity and Sobolev embedding properties of $H^{s,p}$ spaces (ref. \cite{Triebel}). Next, choose a bounded function $F: \R \to \R$ that satisfies $F(x) = \log x$ on $[c,c^{-1}]$ and has bounded continuous derivatives up to order $2$. We will use the fact that for any $s \geq 1$, $1 <p <\infty$ and $f \in C^{[s]+1}(\R)$ that has bounded derivatives up to order $[s]+1$, the composition map $u \mapsto f\circ u$ maps $H^{s,p}(\Omega)\cap H^{1,sp}(\Omega)$ continuously into $H^{s,p}(\Omega)$ \cite{Brezis-comp}. Notice that $H^{3/2,2n}(\Omega) \hookrightarrow H^{1,q}(\Omega)$ for all $q < \infty$ by Sobolev embedding (ref. \cite{Triebel}, Theorem 3.3.1(ii)). Therefore,
			\[
			\|\log \gamma\|_{H^{3/2,2n}(\Omega)} = \|F\circ \gamma\|_{H^{3/2,2n}(\Omega)} < \infty. \]
			Finally, observe that $\log \gamma \in H_0^{3/2,2n}(\Omega)$ (i.e., the closure of $C_c^\infty(\Omega)$ in $H^{3/2,2n}(\Omega)$) and  extension by $0$ is a continuous map from $H_0^{3/2,2n}(\Omega) \to H^{3/2,2n}(\R^n)$ (ref. \cite{Triebel}, Section 3.4.3, Corollary and  Remark 2). Therefore, we get 
			\[
			\|q\|_{H^{-1/2,2n}} \lsim \|\log\gamma\|_{H^{3/2,2n}(\Omega)} +\|\log\gamma\|_{H^{3/2,2n}(\Omega)}^2 < \infty.  \]
			
			\item  Let us first show that $u \in H^1(\Omega)$ if and only if $w = \gamma^{1/2}u \in H^1(\Omega)$. Recall from part (a) that for any bounded smooth function $F:\R \to \R$ with bounded derivatives of all orders, the composition map $u \mapsto F\circ u$ maps $H^{3/2,2n}(\Omega)$ continuously into $H^{3/2,2n}(\Omega)$. Choosing $F$ such that it coincides with $x \mapsto x^{\pm 1/2}$ for $x \in [c,c^{-1}]$, we can infer that $\gamma^{1/2}, \gamma^{-1/2} \in H^{3/2,2n}(\Omega)$ as well. Next, we observe that if $f\in H^{3/2,2n}(\Omega)$ and $g \in H^1(\Omega)$, then $fg \in H^1(\Omega)$ as well. Indeed, since $H^{3/2,2n}(\Omega) \hookrightarrow L^\infty (\Omega)$, it is clear that $fg \in L^2(\Omega)$. Moreover, we have the embeddings $H^{3/2,2n}(\Omega) \hookrightarrow H^{1,n}(\Omega)$ and $H^1(\Omega) \hookrightarrow L^{2n/(n-2)}(\Omega)$ (ref. \cite{Triebel}, Theorem 3.3.1(ii)), and therefore,
			\begin{eqnarray*}
				\|\nabla (fg)\|_{L^2(\Omega)} & \leq & \|f\nabla g\|_{L^2(\Omega)} + \|g\nabla f\|_{L^2(\Omega)} \\
				&\leq & \|f\|_{L^\infty(\Omega)}\|\nabla g\|_{L^2(\Omega)} + \|\nabla f\|_{L^n(\Omega)}\|g\|_{{L^{2n/(n-2)}(\Omega)}} \\
				&\lsim & \|f\|_{H^{3/2,2n}(\Omega)}\|g\|_{H^1(\Omega)} < \infty.
			\end{eqnarray*}
			Consequently, we have the estimate 
			\begin{equation}\label{multh1}
				\|fg\|_{H^1(\Omega)} \lsim \|f\|_{H^{3/2,2n}(\Omega)}\|g\|_{H^1(\Omega)}.
			\end{equation}
			This shows that multiplication by $f\in H^{3/2,2n}(\Omega)$ is a continuous operator on $H^1(\Omega)$. In particular, $w = \gamma^{1/2}u \in H^1(\Omega)$ if and only if $u = \gamma^{-1/2}w \in H^1(\Omega)$. We note here that we can extend the operation of multiplication by $f$ to $H^{-1}(\Omega)$ by duality, i.e., for any $h\in H^{-1}(\Omega)$, define
			$$\langle m_f h, g\rangle := \langle h, fg\rangle, \qquad \textrm{for all } g \in H^1_0(\Omega).$$
			Then \eqref{multh1} immediately implies the estimate
			\begin{equation}\label{multhminus1}
				\|fh\|_{H^{-1}(\Omega)} \lsim \|f\|_{H^{3/2,2n}(\Omega)}\|h\|_{H^{-1}(\Omega)}, \qquad \textrm{for all } f \in H^{3/2,2n}(\Omega), \, h \in H^{-1}(\Omega).
			\end{equation}
			
			Next, we claim that for all $w \in H^1(\Omega)$,
			\begin{equation}\label{transid}	
				\nabla \cdot (\gamma \nabla (\gamma^{-1/2}w)) = \gamma^{1/2}\left(\Delta w -qw\right). 
			\end{equation}
			Indeed, let $\gamma_n$ be a sequence of smooth functions that converge to $\gamma$ in $H^{3/2,2n}(\Omega)$. Then we have
			\begin{eqnarray*}
				\gamma_n \nabla(\gamma_n^{-1/2}w) &=& \gamma_n^{1/2}\nabla w - (\nabla \gamma_n^{1/2})w \quad \textrm{in } L^2(\Omega)\\
				\Rightarrow \nabla \cdot \left(\gamma_n \nabla(\gamma_n^{-1/2}w)\right) &=& \gamma_n^{1/2}\Delta w + \nabla \gamma_n^{1/2}\cdot \nabla w -(\Delta \gamma_n^{1/2})w - \nabla \gamma_n^{1/2}\cdot \nabla w \\
				&=& \gamma_n^{1/2}\Delta w-(\Delta \gamma_n^{1/2})w \quad \textrm{in } H^{-1}(\Omega).
			\end{eqnarray*}
			Observe that as $n\to \infty$, $\gamma_n^{-1/2}w \to \gamma^{-1/2}w$ in $H^1(\Omega)$ by \eqref{multh1}. Consequently, $\gamma_n\nabla(\gamma_n^{-1/2}w) \to \gamma\nabla(\gamma^{-1/2}w)$ in $L^2(\Omega)$ and $\nabla\cdot \left( \gamma_n\nabla(\gamma_n^{-1/2}w)\right) \to \nabla \cdot\left(\gamma\nabla(\gamma^{-1/2}w)\right)$ in $H^{-1}(\Omega)$.
			
			On the right hand side, we have $\gamma_n^{1/2}\Delta w \to \gamma^{1/2}\Delta w$ in $H^{-1}(\Omega)$ by \eqref{multhminus1}. Moreover, $\Delta \gamma_n^{1/2} \to \Delta \gamma^{1/2}$ in $H^{-1/2,2n}(\Omega)$. It follows from Corollary \ref{multiplierbound} that the map $(f,g)\in H^{-1/2,2n}(\Omega)\times H^{1}(\Omega) \mapsto fg \in H^{-1}(\Omega)$ is continuous in both $f$ and $g$ (see also \cite[Corollary 2]{bro03}). Therefore, $(\Delta\gamma_n^{1/2})w \to (\Delta \gamma^{1/2})w$ in $H^{-1}(\Omega)$, and the right hand side
			$$\gamma_n^{1/2}\Delta w - (\Delta \gamma_n^{1/2})w \quad \to \quad \gamma^{1/2}\Delta w - (\Delta \gamma^{1/2})w \quad \textrm{in }H^{-1}(\Omega)$$
			as $n \to \infty$. This proves the identity \eqref{transid}. This along with the fact that $\gamma \equiv 1$ on $\pa \Omega$ implies that $w$ solves \eqref{bvp2} iff $u = \gamma^{-1/2}w$ solves \eqref{bvp1}.
			
			\item $0$ is not a Dirichlet eigenvalue as \eqref{bvp1} and hence \eqref{bvp2} have unique solutions. Now suppose $f \in H^{1/2}(\pa \Omega)$ and $u$ and $w$ are as in \eqref{bvp1} and \eqref{bvp2}. Let $\pa/\pa \nu$ be the outward pointing unit normal vector field on $\pa \Omega$. Since $\gamma \equiv 1$ near $\pa \Omega$,
			\[
			\Lambda_q(f) =\frac{\pa w}{\pa \nu}\Big|_{\pa \Omega} = \gamma \frac{\pa u}{\pa \nu}\Big|_{\pa \Omega} = \Lambda_\gamma(f). \]
		\end{enumerate}
	\end{proof}
	
	Therefore, if we can reconstruct $q = \Delta \sqrt{\gamma}/\sqrt{\gamma}$ from $\Lambda_q=\Lambda_\gamma$, we can reconstruct $\sqrt{\gamma}$ from $q$ as the unique solution of the following boundary value problem:
	
	\begin{prop}
		Let $\gamma$ be as in Theorem \ref{mainth} and $q = \Delta \sqrt{\gamma}/\sqrt{\gamma}$. Then $\sqrt{\gamma}$ is the unique solution in $H^1(\Omega)$ of
		\begin{equation*}
			\left\{
			\begin{array}{rl}
				(-\Delta +q)u=  & 0  \textrm{ in } \Omega \\
				u \equiv & 1 \textrm{ on } \pa \Omega.
			\end{array} \right.
		\end{equation*}
	\end{prop}
	\begin{proof}
		$u = \sqrt{\gamma}$ is clearly a solution. Moreover, the solution is unique by Proposition \ref{reduction}(c) and Proposition \ref{fredholm}(b).
	\end{proof}
	
	In the next two sections, we show how to reconstruct $q$ from $\Lambda_q$. 
	
\section{Complex Geometrical Optics Solutions}
	
	In this section, we will construct CGO solutions to the Schr{\"o}dinger equation $(-\Delta+q)u = 0$ in $\R^n$. The technique used here is standard. See, for example, \cite[Section 4.5]{FSU} for a proof assuming higher regularity on $q$ and $\gamma$.\\
	
	 Observe that if $\zeta \in \C^n$ is such that $\zeta \cdot \zeta = \sum_{j=1}^n \zeta_j^2 = 0$, we have $\Delta e^{x\cdot \zeta} =0$. Viewing $(-\Delta+q)$ as a perturbation of the Laplacian, we look for solutions to $(-\Delta+q)u = 0$ of the form
	\[
	u(x) = e^{x\cdot \zeta}(1+r_\zeta(x)). \]
	Such solutions are called Complex Geometrical Optics (CGO) solutions. We will show the existence of CGO solutions for $|\zeta|$ large enough and establish certain asymptotic bounds on $r_\zeta$ as $|\zeta| \to \infty$. First of all, note that $ u = e^{x\cdot \zeta}(1+r_\zeta)$ solves $(-\Delta+q)u =0$ iff
	\begin{eqnarray}
		-\Delta (e^{x\cdot \zeta}r_\zeta)+e^{x\cdot \zeta}qr_{\zeta} &=& -q \\
		\Leftrightarrow (-\Delta_\zeta +m_q)r_\zeta &=& -q \label{rzeta}
	\end{eqnarray}
	where $\Delta_\zeta v :=e^{-x\cdot \zeta}\Delta(e^{x\cdot \zeta}v) = (\Delta + 2\zeta\cdot \nabla)v$. There exists a right inverse $G_\zeta$ of $\Delta_\zeta$ given by
	\[
	G_\zeta f = \left(\frac{\widehat{f}(\xi)}{-|\xi|^2+2i\zeta\cdot \xi}\right)^{\vee}. \]
	Since the denominator $-|\xi|^2+2i\zeta\cdot \xi$ vanishes only  on a co-dimension 2 sphere in $\R^n$, the right hand side of the above equation is well defined as a tempered distribution whenever $f$ is a Schwartz function (see, for example, \cite[Proposition 4.5.4]{FSU} for details). Looking for solutions of the form $r_\zeta = G_\zeta s_\zeta$ to \eqref{rzeta}, we see that such an $s_\zeta$ should satisfy
	\[
	(I -m_qG_\zeta)s_\zeta = q \]
	where $I$ denotes the identity operator. Our goal is to establish bounds on the operators $m_q$ and $G_\zeta$ between appropriate function spaces such that the operator norm $\|m_qG_\zeta\| < 1$ for $|\zeta|$ large enough. If that is the case, the above equation has a unique solution given by the Neumann series
	\[
	s_\zeta = \sum_{j=0}^\infty (m_q G_\zeta)^j q. \]
	
	Let us begin by recalling some bounds on the operator $G_\zeta$ proved in \cite{SU}.
	\begin{prop}[Sylvester-Uhlmann]\label{syluhl}
		Let $\zeta \in \C^n$  be such that $|\zeta| \geq 1$ and $\zeta \cdot \zeta =0$, and let $0 < \delta <1/2$. Then $G_\zeta$ maps $L^2_{\delta} \to H^2_{-\delta}$ and satisfies the following norm bounds
		\begin{eqnarray*}
			\|G_\zeta u \|_{L^2_{-\delta}} &\lsim & |\zeta|^{-1}\|u\|_{L^2_{\delta}}\\
			\|G_\zeta u \|_{H^1_{-\delta}} &\lsim & \|u\|_{L^2_{\delta}}\\
			\|G_\zeta u \|_{H^2_{-\delta}} &\lsim & |\zeta|\|u\|_{L^2_{\delta}}
		\end{eqnarray*}
		In particular, for $k = |\zeta|$, we have the following scaled estimate:
		\begin{equation*}
			\|G_\zeta u\|_{H^{2}_{-\delta}[k]} \lsim k\|u\|_{L^2_\delta}. 
		\end{equation*}
	\end{prop}
	As an easy corollary, we obtain the following estimate for $G_\zeta$ on negative-order Sobolev spaces:
	\begin{corollary}\label{resolventbound}
		Let $\zeta \in \C^n$ be such that $\zeta \cdot \zeta = 0$ and $k = |\zeta| \geq 1$, and let $0 < \delta < 1/2$. Then $G_\zeta$ maps $H^{-1}_{\delta}[k] \to H^1_{-\delta}(\R^n)[k]$ and satisfies the bound
		\begin{equation}\label{bounddelta}
			\|G_\zeta u\|_{H^1_{-\delta}[k]} \lsim k\|u\|_{H^{-1}_{\delta}[k]}, \qquad u \in H^{-1}_{\delta}[k](\R^n). 
		\end{equation}
	\end{corollary}
	\begin{proof}
		Let $u \in H^{-1}_{\delta}[k](\R^n)$. Then by Proposition \ref{altern}, there exist $u_0,u_1, \ldots ,u_n \in L^2_\delta(\R^n)$ such that $u = u_0 +\sum_{j=1}^n \partial_j u_j$ and
		\[
		k^{-1}\|u_0\|_{L^2_\delta} +\sum_{j=1}^n \|u_j\|_{L^2_\delta} \lsim \|u\|_{H^{-1}_{\delta}[k]}. \]
		Now, by Proposition \ref{syluhl} and the fact that $G_\zeta$ commutes with $\partial_j, j=1, \ldots,n$,
		\begin{eqnarray*}
			\|G_\zeta u_0\|_{L^2_{-\delta}} &\lsim & k^{-1}\|u_0\|_{L^2_\delta} \lsim \|u\|_{H^{-1}_{\delta}[k]}, \\
			\|G_\zeta \partial_j u_j\|_{L^2_{-\delta}} &\lsim & \|G_\zeta u_j\|_{H^1_{-\delta}} \lsim \|u_j\|_{L^2_\delta} \lsim \|u\|_{H^{-1}_{\delta}[k]}, \\
			\|\nabla G_\zeta u_0\|_{L^2_{-\delta}} &\lsim &  \|G_\zeta u_0\|_{H^1_{-\delta}} \lsim \|u_0\|_{L^2_\delta} \lsim k\|u\|_{H^{-1}_{\delta}[k]}, \\
			\|\nabla G_\zeta\partial_j  u_j\|_{L^2_{-\delta}} &\lsim & \|G_\zeta u_j\|_{H^2_{-\delta}} \lsim k\|u_j\|_{L^2_\delta} \lsim k\|u\|_{H^{-1}_{\delta}[k]}. 
		\end{eqnarray*}
		Combining all the above inequalities, we get \eqref{bounddelta}.
	\end{proof}
	
	With the bounds on $m_q$ and $G_\zeta$ in hand, we are now ready to prove the existence of CGO solutions. 
	
	\begin{theorem}\label{CGO}
		Let $q \in H^{-s,n/s}(\R^n)$, $0 < s \leq 1/2$ be such that $\textrm{supp }q$ is compact. Fix $\delta \in (0,1/2)$. Then there exists $M>0$ such that for all $\zeta \in \C^n$ satisfying
		\[
		\zeta \cdot \zeta = 0, \qquad |\zeta| \geq M, \]
		there exists a unique solution to
		\[
		(-\Delta + m_q)u = 0 \quad \textrm{in } \R^n \]
		of the form
		\[
		u = u_\zeta(x) = e^{x\cdot \zeta}(1+r_\zeta(x)) \]
		where $r_\zeta \in H^1_{-\delta}[k](\R^n)$.  Moreover, 
		\[
		\|r_\zeta\|_{H^1_{-\delta}[k]} \lsim |\zeta|^{s}, \]
	\end{theorem}
	\begin{proof}
		As seen before,  $u_\zeta = e^{x\cdot \zeta}(1+r_\zeta)$ satisfies $(-\Delta + q)u = 0$ if and only if
		\[
		(-\Delta_\zeta +q)r_\zeta = -q \]
		where $\Delta_\zeta = e^{-\zeta \cdot x}\Delta e^{\zeta\cdot x}$. We will look for solutions of the form  $r_\zeta = G_\zeta s_\zeta$.  Such an $s_\zeta$ should satisfy
		\begin{equation}\label{neumann}
			(I-m_q \circ G_\zeta)s_\zeta = q. 
		\end{equation}
		Let $k = |\zeta|$. It follows from Corollary \ref{resolventbound} and \eqref{bilinear1delta} from Corollary \ref{multiplierbound} that
		\begin{eqnarray*}
			\|G_\zeta\|_{H^{-1}_{\delta}[k] \to H^1_{-\delta}[k]} &\lsim & k,\\
			\|m_q\|_{H^1_{-\delta}[k] \to H^{-1}_{\delta}[\delta]} &\lsim & k^{-2(1-s)}\omega(k)
		\end{eqnarray*}
		where $\omega(k) \to 0$ as $k \to \infty$. Therefore, $\|m_q \circ G_\zeta\| \lsim k^{-1+2s}\omega(k) \to 0$ as $k \to \infty$ and there exists $M >0$ such that for $k = |\zeta| \geq M$,
		\[
		\|m_q \circ G_\zeta\|_{H^{-1}_{\delta}[k] \to H^{-1}_{\delta}[k]} \leq \frac{1}{2}. \]
		Moreover, $q \in H^{-1}_{\delta}[k](\R^n)$. Indeed, suppose $\varphi \in C_c^\infty(\R^n)$ is such that $\varphi \equiv 1$ on $\textrm{supp }q$. Clearly  $q = q\varphi =  m_q(\varphi)$. Applying Proposition \ref{mV} with $k=1$, we get
		\begin{equation*}
			\|q\|_{H^{-s}} = \|\varphi q\|_{H^{-s}} \lsim \|\varphi\|_{H^{s}} \lsim\|\varphi\|_{H^{s}}.
		\end{equation*}
		Therefore,
		\begin{eqnarray*}
			\|q\|_{H^{-1}_{\delta}[k]} = \|\varphi q\|_{H^{-1}_{\delta}[k]} &\lsim & \|q\|_{H^{-1}[k]} \qquad \textrm{by Lemma \ref{lcompact}}\\
			&\lsim & k^{-(1-s)}\|q\|_{H^{-s}[k]} \\
			&\lsim & k^{-(1-s)}\|q\|_{H^{-s}}\\
			&\lsim & k^{-(1-s)}\|\varphi\|_{H^{s}}.
		\end{eqnarray*}
		Thus, for all $|\zeta| =k \geq M$, \eqref{neumann} has a unique solution given by the Neumann series
		\[
		s_\zeta = \sum_{j=0}^{\infty} (m_q \circ G_\zeta)^j q \]
		and we have the estimates
		\begin{eqnarray}
			\|s_\zeta\|_{H^{-1}_{\delta}[k]} \lsim \|q\|_{H^{-1}_{\delta}[k]} &\lsim & k^{-(1-s)},\\
			\|r_\zeta\|_{H^1_{-\delta}[k]} = \|G_\zeta s_\zeta\|_{H^1_{-\delta}[k]} &\lsim & k^{s}. \label{rbound}
		\end{eqnarray}
		This completes the proof. 
	\end{proof}
	
\section{Uniqueness and Reconstruction}
	Using the integral identity from Proposition \ref{integral} and appropriate CGO solutions, we will be able to reconstruct the Fourier transform of $q$.
	\begin{theorem}
		Let $\Omega$ be a bounded Lipschitz domain in $\R^n$, $n\geq 3$, and $D \subset \Omega$ a compact subset. Let $q \in H^{-1/2,2n}_{D}(\Omega)$ be such that $0$ is not a Dirichlet eigenvalue of $(-\Delta+q)$ in $\Omega$. Let $\xi \in \R^n$ be such that $\xi \neq 0$. Then for $k >0$ sufficiently large, there exist $\zeta_1,\zeta_2 \in \C^n$ with $\zeta_j\cdot \zeta_j =0$ and $|\zeta_j| =k $, $j=1,2,$ such that
		\[
		\lim_{k \to \infty}\langle (\Lambda_q - \Lambda_0)(u_{\zeta_1}|_{\pa \Omega}),e^{x\cdot \zeta_2}\rangle = \langle q,e^{-ix\cdot \xi}\rangle = \int_{\Omega}qe^{-ix\cdot \xi}\,dx \]
		where $\Lambda_0$ denotes the Dirichlet-to Neumann map for $-\Delta$ and $u_{\zeta_1}$ is the unique solution to $(-\Delta+q)u=0$ of the form
		\[
		u_\zeta = e^{x\cdot \zeta}(1+r_\zeta), \qquad r_\zeta \in H^1_{-\delta}[k] \]
		constructed in Theorem \ref{CGO}.
	\end{theorem}
	\begin{proof}
		Let $\alpha, \beta$ be unit vectors in $\R^n$ such that $\{\xi/|\xi|, \alpha, \beta\}$ form an orthonormal set. Define $\zeta_1,\zeta_2 \in \C^n$ by
		\begin{eqnarray}
			\zeta_1 &=& \frac{k}{\sqrt{2}}\alpha +i\left( -\frac{\xi}{2}+\sqrt{\frac{k^2}{2}-\frac{|\xi|^2}{4}}\beta\right),  \label{zeta1}\\
			\zeta_2 &=& -\frac{k}{\sqrt{2}}\alpha +i\left( -\frac{\xi}{2}-\sqrt{\frac{k^2}{2}-\frac{|\xi|^2}{4}}\beta\right).   \label{zeta2}
		\end{eqnarray}
		It is easy to check that $k = |\zeta_1|=|\zeta_2|$ and $\zeta_1\cdot \zeta_1 = \zeta_2\cdot \zeta_2 = 0$. Therefore, by Theorem \ref{CGO}, for $k$ large enough, there exists a solution $u_{\zeta_1} = e^{\zeta_1 \cdot x}(1+r_{\zeta_1}(x))$ of $(-\Delta+q)u = 0$ such that $\|r_{\zeta_1}\|_{H^1_{-\delta}[k]} \lsim k^{1/2}$. Moreover, the fact that $\zeta_2\cdot \zeta_2=0$ implies $\Delta e^{x\cdot \zeta_2} = 0$. Therefore, by Proposition \ref{integral},
		\begin{eqnarray*}
			\langle (\Lambda_q - \Lambda_0)(u_{\zeta_1}|_{\pa \Omega}), e^{x\cdot \zeta_2}\rangle &=& \langle q, u_{\zeta_1}e^{x\cdot \zeta_2}\rangle \\
			&=& \langle q, e^{x\cdot(\zeta_1+\zeta_2)}(1+r_{\zeta_1})\rangle \\
			&=& \langle q, e^{-ix\cdot \xi}\rangle + \langle q, e^{-ix\cdot \xi}r_{\zeta_1}\rangle.
		\end{eqnarray*}
		Now, let $\varphi \in C_c^{\infty}(\R^n)$ be such that $\varphi \equiv 1$ on $\overline{\Omega} \supset \textrm{ supp }q$. By \eqref{bilinear1kdelta} in Corollary \ref{multiplierbound}
		\begin{eqnarray*}
			|\langle q, e^{-ix\cdot \xi}r_{\zeta_1}\rangle| &=& |\langle q, e^{-ix\cdot \xi}\varphi r_{\zeta_1}\rangle|\\
			&\lsim & k^{-1/2}\omega(k)\|e^{-ix\cdot \xi}\varphi\|_{H^1}\|r_{\zeta_1}\|_{H^1_{-\delta}[k]}\\
			&\lsim & k^{-1/2}\omega(k) k^{1/2} = \omega(k) \qquad \textrm{(Theorem \ref{CGO})},
		\end{eqnarray*}
		where $\omega(k) \to 0$ as $k \to \infty$. Therefore, it follows that
		\[
		\lim_{k \to \infty}\langle (\Lambda_q - \Lambda_0)(u_{\zeta_1}|_{\pa \Omega}),e^{x\cdot \zeta_2}\rangle = \langle q,e^{-ix\cdot \xi}\rangle = \widehat{q}(\xi). \]
	\end{proof}
	
	Thus, we see that if $u_\zeta|_{\pa \Omega}$ can somehow be determined, we can recover $\widehat{q}(\xi)$ for $\xi \neq 0$ from the knowledge of $\Lambda_q$. Since $q$ is compactly supported, $\widehat{q}$ is real analytic on $\R^n$ by the Paley-Wiener theorem, and thus $\widehat{q}(0)$ can also be determined by continuity. Therefore, the goal now is to find a procedure to determine $u_\zeta|_{\pa \Omega}$. We will characterize $u_\zeta|_{\pa \Omega}$ as the unique solution of a certain boundary integral equation of Fredholm type. The method is due to Nachman \cite{nachmanrecon}. We will mostly follow the presentation and notation in \cite{FSU}.\\
	
	Let us begin by fixing some notation. We will use $\Omega_+$ to denote the exterior domain $\R^n \setminus \overline{\Omega}$. Let $\Tr: H^1_{\textrm{loc}}(\R^n) \to H^{1/2}(\pa \Omega)$ denote the usual trace operator $\Tr(u) = u|_{\pa \Omega}$. Similarly, we let $\Tr_{+}:H^1(\Omega_+) \to H^{1/2}(\pa \Omega)$ and $\Tr_{-}:H^1(\Omega) \to H^{1/2}$ denote the trace operators in the exterior and interior domains respectively.\\
	
	Let $K_0(x,y) = c_n|x-y|^{2-n}$ be the standard Green's function for the Laplacian. We know that the operator with Schwartz kernel $K_0$ (also denoted by $K_0$) maps $H^{-1}_{D}(\R^n) \to H^1_{\textrm{loc}}(\R^n)$ and satisfies 
	\[
	\Delta K_0 f = f, \qquad f \in H^{-1}_{D}(\R^n). \]
	Now let $\zeta \in \C^n$ be such that $\zeta \cdot \zeta = 0$ and $|\zeta| \geq 1$. We define an analogous operator $K_\zeta$ by
	\[
	K_\zeta(f) =e^{x\cdot \zeta}G_\zeta(e^{-x\cdot \zeta}f). \]
	
	\begin{prop}
		The operator $K_\zeta$ maps $ H^{-1}_{D}(\R^n) \to H^1_{\textrm{loc}}(\R^n)$ and satisfies the following properties:
		\begin{enumerate}[(a)]
			\item $\Delta K_\zeta f = f$ for all $f \in H^{-1}_{D}(\R^n)$. 
			
			\item There exists $R_\zeta \in C^\infty(\R^n \times \R^n)$ such that $K_\zeta = K_0 +R_\zeta$. The operator with Schwartz kernel $R_\zeta$  maps $H^{-m}_{D}(\R^n) \to C^\infty(\R^n)$ for all $m \in \N$.
		\end{enumerate}
	\end{prop}
	\begin{proof}
		Let $ 0 < \delta <1/2$ be arbitrary. Clearly, $f \mapsto e^{-x\cdot \zeta}f$ maps $H^{-1}_{D}$ to $H^{-1}_{D} \hookrightarrow H^{-1}_\delta(\R^n)$. Then by Corollary \ref{resolventbound}, $f \mapsto G_{\zeta}(e^{-x\cdot \zeta}f)$ takes $H^{-1}_{D}$ into $H^1_{-\delta}(\R^n)$. Finally, multiplication by $e^{x\cdot \zeta}$ takes $H^1_{-\delta}(\R^n) \to H^1_{\textrm{loc}}(\R^n)$, which proves that $K_\zeta: H^{-1}_{D}(\R^n) \to H^1_{\textrm{loc}}(\R^n)$.\\
		Now, by definition of $K_\zeta$,
		\[
		\Delta K_\zeta f = e^{x\cdot \zeta}\Delta_\zeta G_{\zeta} (e^{-x\cdot \zeta}f) =f, \quad \forall f \in H^{-1}_{D}(\R^n) \]
		since $G_\zeta$ is a right inverse of $\Delta_\zeta$. This proves (a). Next, define $R_\zeta = K_\zeta - K _0$. Then for any $H^{-1}_{D}(\R^n)$,
		\[
		\Delta R_\zeta f = \Delta K_\zeta f - \Delta K_0f =0. \]
		Therefore, (b) follows from the Elliptic Regularity theorem.
	\end{proof}
	
	\begin{defn}
		The standard Single layer potential is defined as the operator 
		\[
		S_0 = K_0 \Tr^*: H^{-1/2}(\pa \Omega) \to H^1_{\textrm{loc}}(\R^n). \] 
		Analogously, we define the modified (or Fadeev-type) Single layer potential $S_\zeta$ for $\pa \Omega$ by
		\[
		S_\zeta = K_\zeta \Tr^{*} : H^{-1/2}(\pa \Omega) \to H^1_{\textrm{loc}}(\R^n). \]
	\end{defn}
	We will show that $u_\zeta|_{\pa \Omega}$ can be characterized as the unique solution $f \in H^{1/2}(\pa \Omega)$ of the following Boundary Integral Equation:
	\begin{equation}\label{bie}
		(\textrm{Id} +\Tr S_\zeta(\Lambda_q-\Lambda_0))f = e^{x\cdot \zeta} \quad \textrm{ on } \pa \Omega. 
	\end{equation}
	
	\begin{theorem}\label{BIE}
		Let $q \in H^{-1/2,2n}_{D}(\Omega)$ be such that $0$ is not a Dirichlet eigenvalue of $(-\Delta +q)$ in $\Omega$. Let $\zeta \in \C^n$ be such that $\zeta \cdot \zeta =0$ and $|\zeta|$ is sufficiently large, and let $0 <\delta <1/2$. Consider the following problems:
		\begin{eqnarray*}
			&(DE)& \left\{ 
			\begin{array}{l}
				(-\Delta +q)u = 0 \quad \textrm{ in } \R^n, \\
				e^{-x\cdot \zeta}u-1 \in H^1_{-\delta}(\R^n).
			\end{array}
			\right. \\
			&(EP) & \left\{
			\begin{array}{ll}
				(i) & \Delta \widetilde{u} = 0 \quad \textrm{in } \Omega_{+},\\
				(ii) & \widetilde{u} = u|_{\Omega_+} \quad \textrm{for some }u \in H^1_{\textrm{loc}}(\R^n), \\
				(iii) & e^{-x\cdot \zeta}\widetilde{u}-1 = r|_{\Omega_+} \quad \textrm{for some }r \in H^1_{-\delta}(\R^n), \\
				(iv) & (\pa_\nu u)_{+} = \Lambda_q(\Tr_{+}u) \quad \textrm{on } \pa \Omega.
			\end{array}
			\right. \\
			&(BIE)& \left\{ 
			\begin{array}{l}
				(\textrm{Id} + \Tr S_\zeta (\Lambda_q-\Lambda_0))f = e^{x\cdot \zeta} \quad \textrm{ on } \pa \Omega,\\
				f \in H^{1/2}(\pa \Omega).
			\end{array}
			\right.
		\end{eqnarray*}
		Each of these problems has a unique solution. Furthermore, they are equivalent in the following sense: If $u$ solves (DE), $\widetilde{u} = u|_{\Omega_+}$ solves (EP) and conversely, if $\widetilde{u}$ solves (EP), there exists a solution $u$ of (DE) such that $\widetilde{u} = u|_{\Omega_+}$. Also, if $u$ solves (DE), $f := u|_{\pa \Omega}$ solves (BIE) and conversely, if $f$ solves (BIE), there exists a solution $u$ of (DE) such that $f = u|_{\pa \Omega}$.
	\end{theorem}
	\begin{proof}
		(DE) can be rephrased as the problem of finding solutions of the form $u = e^{x\cdot \zeta}(1+r)$ to the equation
		\[
		(-\Delta+q)u = 0 \quad \textrm{in } \R^n, \]
		where $r \in H^1_{-\delta}(\R^n)$. Therefore, (DE) has a unique solution by Theorem \ref{CGO} for $|\zeta|$ sufficiently large. Now we show that (DE) is equivalent to (EP) and (BIE). \\
		
		\noindent \textbf{$(DE) \Rightarrow (BIE)$: } Let $u$ be the solution of (DE) and let $f = u|_{\pa \Omega}$. Clearly, $u \in H^1_{\textrm{loc}}(\R^n)$ and hence $f = \Tr(u) \in H^{1/2}(\pa \Omega)$. Now, fix $x \in \Omega_+$ and define the function $v$ on $\Omega$ by $v(y) = K_\zeta(x,y), \, y \in \Omega$. Since $\Delta v = 0$ in $\Omega$, $v$ is smooth by elliptic regularity. Now, by Green's theorem,
		\[
		\int_{\pa \Omega} (u\pa_\nu v - v\pa_\nu u) d\sigma = \int_\Omega (u\Delta v - v\Delta u). \]
		We know that $\Delta v= 0$ and $\Delta u = qu$. Moreover, since $u,v$ satisfy $(-\Delta +q)u = 0$ and $\Delta v = 0$ in $\Omega$ respectively, $\pa_\nu u = \Lambda_q(u|_{\pa \Omega})$ and $\pa_\nu v = \Lambda_0(v|_{\pa \Omega})$. Substituting these into the above identity, we get
		\begin{eqnarray*}
			\int_{\pa \Omega} u \Lambda_0(v|_{\pa \Omega})\, d\sigma - \int_{\pa \Omega} K_\zeta(x,y) \Lambda_q(f)(y)\, d\sigma(y) &=& -\int_{\Omega} K_\zeta(x,y)(qu)(y)\, dy \\
			\Longrightarrow \int_{\pa \Omega} u \Lambda_0(v|_{\pa \Omega})\, d\sigma -S_{\zeta}\Lambda_qf(x) &=& -K_\zeta(qu)(x).
		\end{eqnarray*}
		Next, by symmetry of $\Lambda_0$, $\int_{\pa \Omega} u\Lambda_0(v|_{\pa \Omega}) \, d\sigma = \int_{\pa \Omega}v\Lambda_0(f)\, d\sigma = S_\zeta \Lambda_0 (f)$. Therefore, the above equation becomes
		\begin{equation}\label{compact}
			S_\zeta(\Lambda_0 - \Lambda_q) f(x) = -K_\zeta(qu)(x), \quad x \in \Omega_{+}. 
		\end{equation}
		Now, we simplify the right hand side. By definition,
		\begin{eqnarray*}
			K_\zeta(qu) &=& e^{x\cdot \zeta}G_\zeta(e^{-x\cdot \zeta}qu) = e^{x\cdot \zeta}G_\zeta(e^{-x\cdot \zeta}\Delta u) \\
			&=& e^{x\cdot \zeta}G_\zeta\circ \Delta_\zeta (e^{-x\cdot \zeta}u) = e^{x\cdot \zeta}G_\zeta\circ \Delta_\zeta (e^{-x\cdot \zeta}u - 1).
		\end{eqnarray*}
		But we know that $e^{-x\cdot \zeta}u - 1 \in H^1_{-\delta}(\R^n)$ and $G_\zeta$ is a right inverse of $\Delta_\zeta$ on $H^1_{-\delta}(\R^n)$. Therefore we get $K_\zeta(qu) = e^{x\cdot \zeta}(e^{-x\cdot \zeta}u-1) =u -e^{x\cdot \zeta}$ and
		\[
		u(x)+S_\zeta(\Lambda_q-\Lambda_0)f (x) = e^{x\cdot \zeta}, \qquad x \in \Omega_{+}. \]
		Taking traces along $\pa \Omega$ on both sides, we get $(\textrm{Id}+\Tr S_\zeta(\Lambda_q-\Lambda_0))f = e^{x\cdot \zeta}$ on $\pa \Omega$, as desired.\\
		
		\noindent $(BIE) \Rightarrow (EP)$: Suppose $f$ solves (BIE). Define 
		\[
		\widetilde{u} := e^{x\cdot \zeta} -S_\zeta(\Lambda_q-\Lambda_0)f. \]
		Clearly, $\widetilde{u}|_{\pa \Omega} = f$ and  $\Delta \widetilde{u} = 0$ on $\R^n \setminus \pa \Omega$. Moreover, (ii) follows from the mapping properties of $S_\zeta$. Next, from the jump properties of single layer potentials, we get
		\[
		(\pa_\nu \widetilde{u})_{-} - (\pa_\nu \widetilde{u})_{+} = -(\Lambda_q-\Lambda_0)f. \]
		Since $\Delta \widetilde{u} = 0$ in $\Omega$, $(\pa_\nu \widetilde{u})_{-} = \Lambda_0(\widetilde{u}|_{\pa \Omega}) = \Lambda_0f$. Therefore, $(\pa_\nu \widetilde{u})_{+} = \Lambda_q f$ and we have verified (iv). Finally, we note that
		\[
		e^{-x\cdot \zeta}\widetilde{u} - 1 = -e^{-x\cdot \zeta}S_\zeta(\Lambda_q-\Lambda_0)f = G_\zeta e^{-x\cdot \zeta}\Tr^* h, \]
		where $h = (\Lambda_0-\Lambda_q)f \in H^{-1/2}(\pa \Omega)$. Since $e^{-x\cdot \zeta}\Tr^* h \in H^{-1}(\R^n)$ is compactly supported, $e^{-x\cdot \zeta}\Tr^*h \in H^{-1}_\delta(\R^n)$ by the usual arguments. Finally, since $G_\zeta : H^{-1}_\delta(\R^n) \to H^1_{-\delta}(\R^n)$, we conclude that $e^{-x\cdot \zeta}\widetilde{u}-1 \in H^1_{-\delta}(\R^n)$.\\
		
		\noindent $(EP) \Rightarrow (DE)$: Let $\widetilde{u}$ solve (EP) and let $v \in H^1(\Omega)$ be the solution of 
		\[
		\left\{
		\begin{array}{rl}
			(-\Delta + q)v =& 0,\\
			v|_{\pa \Omega} =& \gamma_{+}\widetilde{u}.
		\end{array}
		\right.
		\]
		Define $u$ on $\R^n$ by
		\[
		u(x) = \left\{
		\begin{array}{l}
			v(x) \quad \textrm{in } \Omega, \\
			\widetilde{u}(x) \quad \textrm{in} \Omega_{+}.
		\end{array}
		\right.
		\]
		We have $\Tr_{-}(u) = \Tr_{+}(u)$ by construction and $(\pa_\nu u)_{-} =\Lambda_q(\gamma_{+}\widetilde{u}) = (\pa_\nu u)_{+}$ by EP (iv). Therefore, it follows that $u \in H^1_{\textrm{loc}}(\R^n)$ and $(-\Delta+q)u = 0$ in $\R^n$. Finally, $e^{-x\cdot \zeta}u-1 \in H^1_{-\delta}(\R^n)$ because of EP(iii) and the fact that $u = \widetilde{u}$ on $\Omega_{+}$.
	\end{proof}
	
	Let us conclude by showing that the Boundary Integral Equation \eqref{bie} is indeed Fredholm.
	
	\begin{prop}
		Let $q \in H^{-1/2,2n}_{D}(\Omega)$ be such that $0$ is not a Dirichlet eigenvalue of $(-\Delta+m_q)$ on $\Omega$. Then the operator
		\[
		\Tr S_\zeta (\Lambda_q-\Lambda_0): H^{1/2}(\pa \Omega) \to H^{1/2}(\pa \Omega) \]
		is compact. 
	\end{prop}
	\begin{proof}
		Let $P_q: H^{1/2}(\pa \Omega) \to H^1(\Omega)$ be the solution operator that maps $f \in H^{1/2}(\pa \Omega)$ to the unique solution $u \in H^1(\Omega)$ of 
		\[
		\left\{
		\begin{array}{rl}
			(-\Delta + q)u =& 0,\\
			u|_{\pa \Omega} =& f.
		\end{array}
		\right.
		\]
		By the same argument as the one leading to \eqref{compact}, we have
		\[
		\Tr S_\zeta(\Lambda_q-\Lambda_0)f = -\gamma K_{\zeta}\circ m_q \circ P_q (f), \qquad f \in H^{1/2}(\pa \Omega). \]
		But the right hand side is compact since $m_q: H^1(\Omega) \to H^{-1}_{D}(\Omega)$ is compact by Proposition \ref{fredholm}(a). This proves the result.
	\end{proof}
	
	\section{Stability}
	In this final section, we will prove the stability estimates \eqref{gamma-stab} and \eqref{q-stab}. Let us start with the stability estimate for the Schr{\"o}dinger equation. Given $q \in H^{-1/2,2n}_{D}(\Omega)$, we define the set of  Cauchy data for $q$ as
	\[
	\mathcal{C}_q = \left\{ \left(u|_{\pa \Omega}, \frac{\pa u}{\pa \nu}\Big|_{\pa \Omega}\right) \in H^{1/2}(\pa \Omega) \times H^{-1/2}(\pa \Omega): (-\Delta+q)u = 0 \right\}. \]
	If $0$ is not a Dirichlet eigenvalue of $(-\Delta+q)$ on $\Omega$, then $\mathcal{C}_q$ is precisely the graph of the Dirichlet-to-Neumann map $\Lambda_q$. Consider the norm on $H^{1/2}(\pa \Omega) \times H^{-1/2}(\pa \Omega)$ given by
	\[
	\|(f,g)\|_{H^{1/2}\oplus H^{-1/2}} = (\|f\|^2_{H^{1/2}(\pa \Omega)} +\|g\|^2_{H^{-1/2}(\pa \Omega)})^{1/2}. \]
	Given $q_1,q_2 \in H^{-1/2,2n}_{D}(\Omega)$, we define the distance between their Cauchy data sets by 
	\begin{equation*}
		\begin{split}
			\textrm{dist}(\mathcal{C}_{q_1},\mathcal{C}_{q_2}) = \max \Bigg\{ \sup_{(f_1,g_1) \in \mathcal{C}_{q_1}} & \inf_{(f_2,g_2) \in \mathcal{C}_{q_2}}\frac{\|(f_1-f_2,g_1-g_2)\|_{H^{1/2}\oplus H^{-1/2}}}{\|(f_1,g_1)\|_{H^{1/2}\oplus H^{-1/2}}},\\ 
			& \sup_{(f_2,g_2) \in \mathcal{C}_{q_2}} \inf_{(f_1,g_1) \in \mathcal{C}_{q_1}}\frac{\|(f_1-f_2,g_1-g_2)\|_{H^{1/2}\oplus H^{-1/2}}}{\|(f_2,g_2)\|_{H^{1/2}\oplus H^{-1/2}}} \Bigg\}.
		\end{split}  
	\end{equation*}
	It can be verified that if $\mathcal{C}_{q_j}$ are in fact the graphs of the Dirichlet-to-Neumann maps $\Lambda_{q_j}$,
	\begin{equation}
		\frac{\|\Lambda_{q_1}-\Lambda_{q_2}\|_{H^{1/2} \to H^{-1/2}}}{\sqrt{1+\|\Lambda_{q_1}\|_{H^{1/2} \to H^{-1/2}}^2}\sqrt{1+\|\Lambda_{q_2}\|_{H^{1/2} \to H^{-1/2}}^2}} \leq \textrm{dist}(\mathcal{C}_{q_1},\mathcal{C}_{q_2}) \leq \|\Lambda_{q_1}-\Lambda_{q_2}\|_{H^{1/2} \to H^{-1/2}}. 
	\end{equation}
	We will establish bounds on $\|q_1-q_2\|_{H^{-1}}$ in terms of $\textrm{dist}(\mathcal{C}_{q_1},\mathcal{C}_{q_2})$, thus including the cases where $0$ is a Dirichlet eigenvalue of one of $(-\Delta+q_j)|_\Omega$. The estimate \eqref{q-stab} follows from the theorem below:
	\begin{theorem}\label{stab-schrodinger}
		Let $0 <s <1/2$ and $q_1,q_2 \in H^{-s,n/s}_{D}(\Omega)$ satisfy the a priori estimate
		\[
		\|q_j\|_{H^{-s,n/s}} \leq M, \qquad j=1,2. \]
		Then there exists $C = C(\Omega, n, s, M)>0$ (independent of $D$) and $\sigma =\sigma(n,s) \in (0,1)$ such that
		\begin{equation}\label{q-stab2}
			\|q_1-q_2\|_{H^{-1}} \leq C(|\log\{\textrm{dist}(\mathcal{C}_{q_1},\mathcal{C}_{q_2})\}|^{-\sigma} +\textrm{dist}(\mathcal{C}_{q_1},\mathcal{C}_{q_2})). 
		\end{equation}
	\end{theorem}
	\begin{proof}
		Let $u_1, u_2 \in H^1(\Omega)$ satisfy $(-\Delta+q_j)u_j = 0 $ in $\Omega$, $j=1,2$. By the weak definition of normal trace, we have
		\begin{eqnarray*}
			\int_{\pa \Omega} \left(u_2\frac{\pa u_1}{\pa \nu}-u_1\frac{\pa u_2}{\pa \nu}\right)\, d\sigma &=& \int_\Omega (\nabla u_1\cdot \nabla u_1 +q_1u_1u_2) \, dx\\
			& &  -\int_\Omega (\nabla u_2\cdot \nabla u_1 +q_2u_1u_2 )\, dx \\
			&=& \int_\Omega (q_1-q_2)u_1 u_2 \, dx.
		\end{eqnarray*}
		Suppose $(f,g) \in \mathcal{C}_{q_1}$. Then there exists $v \in H^1(\Omega)$ such that $(-\Delta +q_1)v =0$, and
		\[
		v|_{\pa \Omega} = f, \qquad \frac{\pa v}{\pa \nu}\Big|_{\pa \Omega} = g. \]
		By the same argument as above,
		\[
		0 = \int_{\pa \Omega} (q_1-q_1)u_1 v\, dx = \int_{\pa \Omega} \left(f\frac{\pa u_1}{\pa \nu}-u_1g\right)\, d\sigma \]
		and therefore,
		\[
		\int_{\Omega} (q_1-q_2)u_1u_2\, dx = \int_{\pa \Omega} \left( (u_2-f)\frac{\pa u_1}{\pa \nu} - u_1\left(\frac{\pa u_2}{\pa \nu}-g\right)\right)\, d\sigma. \]
		This implies
		\begin{eqnarray*}
			\left|\int_\Omega (q_1-q_2)u_1u_2\, dx \right| &\leq & \|u_2-f\|_{H^{1/2}(\pa \Omega)}\left\|\frac{\pa u_1}{\pa \nu}\right\|_{H^{-1/2}(\pa \Omega)} + \|u_1\|_{H^{1/2}(\pa \Omega)}\left\|\frac{\pa u_2}{\pa \nu}-g\right\|_{H^{-1/2}(\pa \Omega)}\\
			&\leq & \left\|\left(u_1, \frac{\pa u_1}{\pa \nu}\right)\right\|_{H^{1/2} \oplus H^{-1/2}}\cdot \left\|\left(u_2-f, \frac{\pa u_2}{\pa \nu}-g\right)\right\|_{H^{1/2}\oplus H^{-1/2}}.
		\end{eqnarray*}
		Taking supremum over all $(f,g) \in \mathcal{C}_{q_1}$,
		\begin{equation} \label{uwu1}
			\left|\int_\Omega (q_1-q_2)u_1u_2\, dx \right|  \leq \left\|\left(u_1, \frac{\pa u_1}{\pa \nu}\right)\right\|\cdot \textrm{dist}(\mathcal{C}_{q_1},\mathcal{C}_{q_2}) \cdot \left\|\left(u_2, \frac{\pa u_2}{\pa \nu}\right)\right\|. 
		\end{equation}
		Now, we let $u_1,u_2$ be the CGO solutions constructed in Theorem \ref{CGO}. Choose $k >0, \, \xi \in \R^n \setminus \{0\}$ and let $\alpha, \beta$ be unit vectors in $\R^n$ such that $\{\alpha, \beta, \xi/|\xi|\}$ forms an orthonormal set. Define $\zeta_1, \zeta_2 \in \C^n$ as in \eqref{zeta1}-\eqref{zeta2} and let
		\begin{eqnarray*}
			u_1(x) &=& u_{\zeta_1}(x) = e^{x\cdot \zeta_1}(1+r_1(x)), \\
			u_2(x) &=& u_{\zeta_2}(x) = e^{x\cdot \zeta_2}(1+r_2(x)).
		\end{eqnarray*}
		where $r_j$, $j=1,2$,  satisfy \eqref{rbound}. It follows that 
		\begin{eqnarray*}
			\left\|\left(u_j, \frac{\pa u_j}{\pa \nu}\right)\right\|_{H^{1/2}\oplus H^{-1/2}} &\lsim & \|u_j\|_{H^1(\Omega)} \lsim \|e^{x\cdot \zeta_j}\|_{C^1(\Omega)}\|1+r_j\|_{H^1(\Omega)}\\
			& \lsim & ke^{Rk}(1+k^{s}) \qquad \textrm{where } R = \sup_{x \in \Omega}|x| \\
			&\lsim & e^{Sk}, \qquad \textrm{for some }S >R.
		\end{eqnarray*}
		Substituting in \eqref{uwu1}, we get
		\[
		\left|\int_\Omega (q_1-q_2)u_1u_2\, dx \right| \lsim e^{2Sk}\textrm{dist}(\mathcal{C}_{q_1},\mathcal{C}_{q_2}). \]
		Now consider
		\[
		(\widehat{q}_1-\widehat{q}_2)(\xi) =  \int_\Omega (q_1-q_2)e^{-ix\cdot \xi}\, dx = \int_\Omega (q_1-q_2)(u_1u_2 -e^{-ix\cdot \xi}(r_1+r_2+r_1r_2))\, dx. \]
		This implies
		\begin{equation}\label{uwu2}
			\begin{split}
				|(\widehat{q}_1-\widehat{q}_2)(\xi)| \lsim & \left| \int_\Omega (q_1-q_2)u_1u_2\, dx \right| +|\langle q_1-q_2, e^{-ix\cdot \xi}(r_1+r_2)\rangle|\\
				&+|\langle m_{q_1-q_2}(e^{-ix\cdot \xi}r_1),r_2\rangle|.
			\end{split}
		\end{equation}
		Choose a cut-off function $\varphi \in C_c^\infty(\R^n)$ such that $\varphi \equiv 1$ on $\overline{\Omega}$. By \eqref{sbound}, 
		\[
		|\langle m_{q_j}(e^{-ix\cdot \xi}r_1), r_2 \rangle | \lsim \omega(k)\|e^{-ix\cdot \xi}\varphi r_1\|_{H^s[k]}\|\varphi r_2\|_{H^s[k]} \]
		where $\omega(k) \to 0$ as $k \to \infty$. It is obvious from the proof of Proposition \ref{mV} that
		\[
		\omega(k) \leq \max_{j=1,2}\|q_j\|_{H^{-s,n/s}} \leq M \qquad \textrm{for all }k \geq 1. \]
		Also, for any $f \in H^s[k](\R^n)$,
		\begin{eqnarray*}
			\|e^{-ix\cdot \xi}f\|_{H^s[k]}^2 &=& \frac{1}{(2\pi)^n}\int |\widehat{e^{-ix\cdot \xi}f}(\eta)|^2(k^2+|\eta|^2)^s \, d\eta \\
			&=& \frac{1}{(2\pi)^n}\int |\widehat{f}(\eta+\xi)|^2(k^2+|\eta|^2)^s \, d\eta \\
			&=& \frac{1}{(2\pi)^n} \int |\widehat{f}(\eta)|^2 k^{2s}\left(1 +\frac{|\eta-\xi|^2}{k^2}\right)^s \, d\eta \\
			&\lsim & \int |\widehat{f}(\eta)|^2 k^{2s}\left(1+\frac{|\eta|^2}{k^2}\right)^s\left(1+\frac{|\xi|^2}{k^2}\right)^s \, d\eta \quad \textrm{(Peetre's inequality)}\\
			&\lsim & (1+|\xi|^2)^s \int |\widehat{f}(\eta)|^2 (k^2+|\eta|^2)^s \, d\eta \\
			&\lsim & (1+|\xi|^2)^s\|f\|_{H^s[k]}^2.
		\end{eqnarray*}
		Therefore,
		\begin{eqnarray*}
			|\langle m_{q_j}(e^{-ix\cdot \xi}r_1, r_2 \rangle | &\lsim & M(1+|\xi|^2)^{s/2}\|\varphi r_1\|_{H^s[k]}\|\varphi r_2\|_{H^s[k]} \\
			&\lsim & k^{-2(1-s)}M(1+|\xi|^2)^{s/2}\|r_1\|_{H^1_{-\delta}[k]}\|r_2\|_{H^1_{-\delta}[k]}\\
			&\lsim & k^{-4\epsilon}M(1+|\xi|^2)^{s/2} \qquad \textrm{by }\eqref{rbound}.
		\end{eqnarray*}
		Next, again by \eqref{sbound}, for $j,l=1,2$,
		\begin{eqnarray*}
			|\langle q_j, e^{-ix\cdot \xi}r_l \rangle | &=& |\langle m_{q_j}(\varphi), e^{-ix\cdot \xi}\varphi r_l \rangle | \\
			&\lsim & \omega(k)\|\varphi\|_{H^{s}}\|e^{-ix\dot\xi}\varphi r_l\|_{H^s[k]} \lsim M(1+|\xi|^2)^{s/2}\|\varphi r_l \|_{H^s[k]} \\
			&\lsim &  M(1+|\xi|^2)^{s/2}k^{-1+s}\|r_l\|_{H^1_{-\delta}[k]}\\
			&\lsim & M(1+|\xi|^2)^{s/2}k^{-2\epsilon} \qquad \textrm{by }\eqref{rbound}.
		\end{eqnarray*}
		Substituting all these bounds into \eqref{uwu2}, we get
		\[
		|\widehat{q}_1(\xi)-\widehat{q}_2(\xi)| \lsim e^{2Sk}\textrm{dist}(\mathcal{C}_{q_1},\mathcal{C}_{q_2}) + k^{-2\epsilon}M(1+|\xi|^2)^{s/2}. \]
		We therefore have
		\begin{eqnarray*}
			\|q_1-q_2\|^2_{H^{-1}} &=& \frac{1}{(2\pi)^n}\int_{\R^n} (1+|\xi|^2)^{-1}|\widehat{q}_1(\xi)-\widehat{q}_2(\xi)|^2\, d\xi \\
			&\lsim & \int_{|\xi| \leq \rho} (1+|\xi|^2)^{-1}|\widehat{q}_1(\xi)-\widehat{q}_2(\xi)|^2\, d\xi +\int_{|\xi|>\rho} (1+|\xi|^2)^{-1}|\widehat{q}_1(\xi)-\widehat{q}_2(\xi)|^2\, d\xi \\
			&\lsim & \rho^n e^{4Sk}\textrm{dist}^2(\mathcal{C}_{q_1},\mathcal{C}_{q_2})+k^{-4\epsilon}M^2\rho^{n+2s-2}\\
			& & +\frac{1}{(1+\rho^2)^{1-s}}\int (1+|\xi|^2)^{-s}(\widehat{q}^2_1(\xi)+\widehat{q}^2_2(\xi))\, d\xi\\
			&\lsim & \rho^{n}e^{4Sk}\textrm{dist}^2(\mathcal{C}_{q_1},\mathcal{C}_{q_2}) + M^2 k^{-4\eps}\rho^{n-2\epsilon-1}+M^2\rho^{-1-2\epsilon}.
		\end{eqnarray*}
		In order to make the last two terms small and of the same order in $\rho$, we choose
		\[
		k = \rho^{\frac{n}{4\epsilon}}, \]
		which gives us
		\begin{eqnarray}
			\|q_1-q_2\|^2_{H^{-1}} &\lsim & \rho^n e^{4S\rho^{\frac{n}{4\epsilon}}}\textrm{dist}^2(\mathcal{C}_{q_1},\mathcal{C}_{q_2}) +\rho^{-1-2\epsilon} \\
			&\lsim & e^{T\rho^{\frac{n}{4\epsilon}}}\textrm{dist}^2(\mathcal{C}_{q_1},\mathcal{C}_{q_2}) +\rho^{-1-2\epsilon} \label{penu}
		\end{eqnarray}
		for fixed $T > 4S$. Now choose
		\[
		\rho = \left( \frac{1}{T}|\log\{\textrm{dist}(\mathcal{C}_{q_1},\mathcal{C}_{q_2})\}|\right)^{\frac{4\epsilon}{n}} \]
		so that when $\textrm{dist}(\mathcal{C}_{q_1},\mathcal{C}_{q_2}) < 1$,
		\[
		e^{T\rho^{\frac{n}{4\epsilon}}}\textrm{dist}^2(\mathcal{C}_{q_1},\mathcal{C}_{q_2}) = \textrm{dist}(\mathcal{C}_{q_1},\mathcal{C}_{q_2}). \]
		Combining this with \eqref{penu}, we see that when $\textrm{dist}(\mathcal{C}_{q_1},\mathcal{C}_{q_2}) <1$,
		\begin{eqnarray*}
			\|q_1-q_2\|^2_{H^{-1}} &\lsim & \textrm{dist}(\mathcal{C}_{q_1},\mathcal{C}_{q_2}) + |\log\{\textrm{dist}(\mathcal{C}_{q_1},\mathcal{C}_{q_2})\}|^{-\frac{4\epsilon(1+2\epsilon)}{n}} \\
			&\lsim & |\log\{\textrm{dist}(\mathcal{C}_{q_1},\mathcal{C}_{q_2})\}|^{-\frac{4\epsilon(1+2\epsilon)}{n}}.
		\end{eqnarray*}
		This gives us \eqref{q-stab2} when $\textrm{dist}(\mathcal{C}_{q_1},\mathcal{C}_{q_2})<1$ for $\sigma = 4\epsilon(1+2\epsilon)/n = 4(1-s)(1-2s)/n$. Moreover, \eqref{q-stab2} is trivially true when $\textrm{dist}(\mathcal{C}_{q_1},\mathcal{C}_{q_2})>1$ since $\|q_j\|_{H^{-1}} \lsim \|q_j\|_{H^{-s,n/s}} \leq M$ for $j=1,2$. Therefore, the proof is complete.
	\end{proof}
	We can now prove the stability estimate for the conductivity equation. We will use the fact that $H^{s,p}$ embeds into the Zygmund space $C^t_*$ for $t = s-n/p$.
	
	\begin{theorem}
		Let $0 < s < 1/2$ and $\gamma_1, \gamma_2 \in H^{2-s,n/s}(\Omega)$ be such that $\gamma_j \equiv 1$ in $\Omega \setminus D$ and 
		\[
		0 < c < \gamma_j(x) < c^{-1}, \quad \textrm{for a.e. } x \in \Omega, \,  j=1,2. \]
		Given any $\alpha \in (0,1)$, there exists $C=C(\Omega,n,s,c,M,\alpha)>0$ (independent of $D$) and $\sigma = \sigma(n,s,\alpha) \in (0,1)$ such that
		\begin{equation}\label{xyz}
			\|\gamma_1-\gamma_2\|_{C^\alpha(\overline{\Omega})} \leq C(|\log \|\Lambda_{\gamma_1}-\Lambda_{\gamma_2}\|_{H^{1/2}\to H^{-1/2}}|^{-\sigma} +\|\Lambda_{\gamma_1}-\Lambda_{\gamma_2}\|_{H^{1/2}\to H^{-1/2}}).
		\end{equation}
	\end{theorem}
	\begin{proof}
		As in Proposition \ref{reduction}, let us extend $\gamma_j$ to all of $\R^n$ by defining $\gamma_j \equiv 1$ on $\R^n \setminus \Omega$, so that $\gamma_j - 1 \in H^{2-s,n/s}_{D}(\Omega)$. Note that this implies $\gamma_j \in C^{1+\epsilon}_* = C^{1,\epsilon}$ for $\epsilon = 1-2s$.  Define $q_j = \gamma_j^{-1/2}\Delta \gamma_j^{1/2}$. Also choose a bounded domain $U$ such that $\overline{\Omega} \subset U$ and $\pa U$ is smooth. We observe that the function $v = \log \gamma_1 -\log \gamma_2$ solves the following elliptic boundary value problem:
		\begin{equation*}
			\left\{
			\begin{array}{rll}
				\nabla \cdot ((\gamma_1\gamma_2)^{1/2}\nabla v)= & 2(\gamma_1 \gamma_2)^{1/2}(q_2-q_1) &  \textrm{ in } U \\
				v= & 0 &\textrm{ on } \partial U.
			\end{array} \right.
		\end{equation*}
		Therefore, we have the estimate
		\begin{equation*}
			\|\log \gamma_1 - \log \gamma_2\|_{H^1(U)} \lsim \|q_1-q_2\|_{H^{-1}(U)} \lsim \|q_1-q_2\|_{H^{-1}}.
		\end{equation*}
		Now consider the identities
		\begin{eqnarray*}
			\gamma_1 -\gamma_2 &=& \left(\int_0^1 e^{t\log \gamma_1 +(1-t)\log \gamma_2}dt\right)\cdot (\log \gamma_1 -\log \gamma_2), \\
			\nabla \gamma_1 - \nabla \gamma_2 &=& \gamma_1 \nabla \log \gamma_1 - \gamma_2 \nabla \log \gamma_2 = \gamma_1(\nabla \log \gamma_1 -\log \gamma_2) +\frac{\gamma_1-\gamma_2}{\gamma_2}\nabla \gamma_2.
		\end{eqnarray*}
		Together with the fact that $\gamma_j \in C^{1,\epsilon}$, these identities imply that
		\begin{equation}\label{fvd1}
			\|\gamma_1-\gamma_2\|_{H^1(U)} \lsim \|\log \gamma_1 -\log \gamma_2\|_{H^1(U)} \lsim \|q_1-q_2\|_{H^{-1}}. 
		\end{equation}
		Next, recall from Proposition \ref{reduction}(c) that $\Lambda_{\gamma_j} = \Lambda_{q_j}$. By \eqref{q-stab2}, for $\|\Lambda_{\gamma_1}-\Lambda_{\gamma_2}\|_{H^{1/2} \to H^{-1/2}} = \|\Lambda_{q_1}-\Lambda_{q_2}\|_{H^{1/2}\to H^{-1/2}} < 1/2$,
		\[
		\|q_1-q_2\|_{H^{-1}} \lsim |\log \|\Lambda_{\gamma_1}-\Lambda_{\gamma_2}\|_{H^{1/2}\to H^{-1/2}}|^{-\sigma} \]
		which along with \eqref{fvd1} implies
		\begin{equation*}
			\|\gamma_1-\gamma_2\|_{H^1(U)} \lsim |\log \|\Lambda_{\gamma_1}-\Lambda_{\gamma_2}\|_{H^{1/2}\to H^{-1/2}}|^{-\sigma}.
		\end{equation*}
		Now, given $\alpha \in (0,1)$, define $p = n/(1-\alpha)$.  By H{\" o}lder's inequality and the fact that $\gamma_j, \nabla \gamma_j$ are bounded,
		\[
		\|\gamma_1-\gamma_2\|_{H^{1,p}(U)} \lsim \|\gamma_1-\gamma_2\|_{H^1(U)}^{2/p}. \]
		Therefore, whenever $\|\Lambda_{\gamma_1}-\Lambda_{\gamma_2}\|_{H^{1/2} \to H^{-1/2}} < 1/2$,
		\begin{eqnarray*}
			\|\gamma_1-\gamma_2\|_{H^{1,p}(U)} &\lsim & |\log \|\Lambda_{\gamma_1}-\Lambda_{\gamma_2}\|_{H^{1/2}\to H^{-1/2}}|^{-\sigma'}\\
			&\lsim & |\log \|\Lambda_{\gamma_1}-\Lambda_{\gamma_2}\|_{H^{1/2}\to H^{-1/2}}|^{-\sigma'} +\|\Lambda_{\gamma_1}-\Lambda_{\gamma_2}\|_{H^{1/2} \to H^{-1/2}}
		\end{eqnarray*}
		for $\sigma' = \frac{2\sigma}{p} = \frac{8(1-s)(1-2s)(1-\alpha)}{n^2}$. On the other hand, the above estimate is clearly true when $\|\Lambda_{\gamma_1}-\Lambda_{\gamma_2}\|_{H^{1/2} \to H^{-1/2}} \geq 1/2$ due to the fact that $\gamma_j \in H^{1,\infty}$. Therefore, in all cases, we have
		\[
		\|\gamma_1-\gamma_2\|_{H^{1,p}(U)} \lsim |\log \|\Lambda_{\gamma_1}-\Lambda_{\gamma_2}\|_{H^{1/2}\to H^{-1/2}}|^{-\sigma'} +\|\Lambda_{\gamma_1}-\Lambda_{\gamma_2}\|_{H^{1/2} \to H^{-1/2}}. \]
		Finally, \eqref{xyz} follows from the fact that $H^{1,p}(U) \hookrightarrow C^\alpha_*(U) = C^\alpha(U) \hookrightarrow C^\alpha(\overline{\Omega})$.
	\end{proof}
	
	\section*{Acknowledgments}
	
	The author would like to thank Gunther Uhlmann and Hart Smith for their helpful suggestions and comments on the manuscript, as well as their mentorship throughout this project. This project was initiated and largely completed during the author's tenure as a Predoctoral Teaching and Research Assistant in the Department of Mathematics at the University of Washington. At the University of Jyv{\"a}skyl{\"a}, the author was supported by the Research Council of Finland (Flagship of Advanced Mathematics for Sensing Imaging and Modelling grant 359208; Centre of Excellence of Inverse Modelling and Imaging grant 353092; and other grants 358047, 360434).
	
\bibliography{Master_bibfile} 

\def\cprime{$'$}
\begin{thebibliography}{10}

\bibitem{alessandrini88}
Giovanni Alessandrini.
\newblock Stable determination of conductivity by boundary measurements.
\newblock {\em Applicable Analysis}, 27(1-3):153--172, 1988.

\bibitem{alessandrini90}
Giovanni Alessandrini.
\newblock Singular solutions of elliptic equations and the determination of
  conductivity by boundary measurements.
\newblock {\em Journal of Differential Equations}, 84(2):252 -- 272, 1990.

\bibitem{alessandrini91}
Giovanni Alessandrini.
\newblock {\em Determining Conductivity by Boundary Measurements, the Stability
  Issue}, pages 317--324.
\newblock Springer Netherlands, Dordrecht, 1991.

\bibitem{ammari-partial}
Habib Ammari and Gunther Uhlmann.
\newblock Reconstruction of the potential from partial {C}auchy data for the
  {S}chr\"{o}dinger equation.
\newblock {\em Indiana Univ. Math. J.}, 53(1):169--183, 2004.

\bibitem{yernat-partial}
Yernat~M. Assylbekov.
\newblock Reconstruction in the partial data {C}alder\'{o}n problem on
  admissible manifolds.
\newblock {\em Inverse Probl. Imaging}, 11(3):455--476, 2017.

\bibitem{Paivarinta}
Kari Astala and Lassi P\"aiv\"arinta.
\newblock Calder\'on's inverse conductivity problem in the plane.
\newblock {\em Ann. of Math. (2)}, 163(1):265--299, 2006.

\bibitem{Blaasten}
Eemeli Blaasten, Oleg~Yu. Imanuvilov, and M.~Yamamoto.
\newblock Stability and uniqueness for a two-dimensional inverse boundary value
  problem for less regular potentials.
\newblock 2015.

\bibitem{Brezis-comp}
Haïm Brezis and Petru Mironescu.
\newblock Gagliardo-nirenberg, composition and products in fractional sobolev
  spaces.
\newblock {\em Journal of Evolution Equations}, 1(4):387--404, 2001.

\bibitem{brown96}
R.M. Brown.
\newblock Global uniqueness in the impedance-imaging problem for less regular
  conductivities.
\newblock {\em SIAM J. Math. Anal.}, 27(4):1049--1056, July 1996.

\bibitem{bro03}
Russell~M. Brown and Rodolfo~H. Torres.
\newblock Uniqueness in the inverse conductivity problem for conductivities
  with $3/2$ derivatives in $l^p, p >2n$.
\newblock {\em Journal of Fourier Analysis and Applications}, 9(6):563--574,
  Nov 2003.

\bibitem{BrownUhlmann}
Russell~M. Brown and Gunther~A. Uhlmann.
\newblock Uniqueness in the inverse conductivity problem for nonsmooth
  conductivities in two dimensions.
\newblock {\em Comm. Partial Differential Equations}, 22(5-6):1009--1027, 1997.

\bibitem{bukhgeim}
A.~L. Bukhgeim.
\newblock Recovering a potential from {C}auchy data in the two-dimensional
  case.
\newblock {\em J. Inverse Ill-Posed Probl.}, 16(1):19--33, 2008.

\bibitem{bukhgeim-uhlmann}
Alexander~L. Bukhgeim and Gunther Uhlmann.
\newblock Recovering a potential from partial {C}auchy data.
\newblock {\em Comm. Partial Differential Equations}, 27(3-4):653--668, 2002.

\bibitem{Cal80}
A.P. Calderon.
\newblock On an inverse boundary value problem.
\newblock {\em Seminar on Numerical Analysis and its Applications to Continuum
  Physics}, pages 65--73, 1980.

\bibitem{caro-partial}
Pedro Caro, David Dos Santos~Ferreira, and Alberto Ruiz.
\newblock Stability estimates for the {C}alder\'{o}n problem with partial data.
\newblock {\em J. Differential Equations}, 260(3):2457--2489, 2016.

\bibitem{Caro-Stability}
Pedro Caro, Andoni García, and Juan~Manuel Reyes.
\newblock Stability of the calderón problem for less regular conductivities.
\newblock {\em Journal of Differential Equations}, 254(2):469 -- 492, 2013.

\bibitem{carogarcia}
Pedro Caro, María~Ángeles García-Ferrero, and Keith~M. Rogers.
\newblock Reconstruction for the calderón problem with lipschitz
  conductivities.
\newblock {\em Analysis \& PDE}, 18(8):2033–2060, July 2025.

\bibitem{Caro}
Pedro Caro, Tapio Helin, and Matti Lassas.
\newblock Inverse scattering for a random potential.
\newblock {\em Analysis and Applications}, 0(0):1--55, 0.

\bibitem{caro-rogers}
Pedro Caro and Keith~M. Rogers.
\newblock Global uniqueness for the calder{\'o}n problem with lipschitz
  conductivities.
\newblock {\em Forum of Mathematics, Pi}, 4:e2, 2016.

\bibitem{FSU}
Joel~S. Feldman, Mikko Salo, and Gunther Uhlmann.
\newblock {\em The Calderón problem : An Introduction}.
\newblock Graduate Studies in Mathematics. American Mathematical Society,
  Providence, Rhode Island, 2025.

\bibitem{Garcia_2016}
Andoni Garc{\'{\i}}a and Guo Zhang.
\newblock Reconstruction from boundary measurements for less regular
  conductivities.
\newblock {\em Inverse Problems}, 32(11):115015, oct 2016.

\bibitem{katoponce-grafakos}
Loukas Grafakos and Seungly Oh.
\newblock The {K}ato-{P}once inequality.
\newblock {\em Comm. Partial Differential Equations}, 39(6):1128--1157, 2014.

\bibitem{haberman}
Boaz Haberman.
\newblock Uniqueness in {C}alder\'on's problem for conductivities with
  unbounded gradient.
\newblock {\em Comm. Math. Phys.}, 340(2):639--659, 2015.

\bibitem{HT13}
Boaz Haberman and Daniel Tataru.
\newblock Uniqueness in calderón’s problem with lipschitz conductivities.
\newblock {\em Duke Mathematical Journal}, 162(3):497–516, Feb 2013.

\bibitem{Heck-Stability}
Horst Heck.
\newblock Stability estimates for the inverse conductivity problem for less
  regular conductivities.
\newblock {\em Communications in Partial Differential Equations},
  34(2):107--118, 2009.

\bibitem{heck-partial}
Horst Heck and Jenn-Nan Wang.
\newblock Stability estimates for the inverse boundary value problem by partial
  {C}auchy data.
\newblock {\em Inverse Problems}, 22(5):1787--1796, 2006.

\bibitem{Imanuvilov3}
Oleg~Y. Imanuvilov and Masahiro Yamamoto.
\newblock Uniqueness for inverse boundary value problems by
  {D}irichlet-to-{N}eumann map on subboundaries.
\newblock {\em Milan J. Math.}, 81(2):187--258, 2013.

\bibitem{Imanuvilov1}
Oleg~Yu. Imanuvilov, Gunther Uhlmann, and Masahiro Yamamoto.
\newblock The {C}alder\'{o}n problem with partial data in two dimensions.
\newblock {\em J. Amer. Math. Soc.}, 23(3):655--691, 2010.

\bibitem{Imanuvilov2}
Oleg~Yu. Imanuvilov, Gunther Uhlmann, and Masahiro Yamamoto.
\newblock Inverse boundary value problem by measuring {D}irichlet data and
  {N}eumann data on disjoint sets.
\newblock {\em Inverse Problems}, 27(8):085007, 26, 2011.

\bibitem{Isakov-partial}
Victor Isakov.
\newblock On uniqueness in the inverse conductivity problem with local data.
\newblock {\em Inverse Probl. Imaging}, 1(1):95--105, 2007.

\bibitem{katoponce-orig}
Tosio Kato and Gustavo Ponce.
\newblock Commutator estimates and the {E}uler and {N}avier-{S}tokes equations.
\newblock {\em Comm. Pure Appl. Math.}, 41(7):891--907, 1988.

\bibitem{Kenig-Salo-I}
Carlos Kenig and Mikko Salo.
\newblock The {C}alder\'{o}n problem with partial data on manifolds and
  applications.
\newblock {\em Anal. PDE}, 6(8):2003--2048, 2013.

\bibitem{Kenig-Salo-II}
Carlos Kenig and Mikko Salo.
\newblock Recent progress in the {C}alder\'{o}n problem with partial data.
\newblock In {\em Inverse problems and applications}, volume 615 of {\em
  Contemp. Math.}, pages 193--222. Amer. Math. Soc., Providence, RI, 2014.

\bibitem{KSjU}
Carlos~E. Kenig, Johannes Sj\"{o}strand, and Gunther Uhlmann.
\newblock The {C}alder\'{o}n problem with partial data.
\newblock {\em Ann. of Math. (2)}, 165(2):567--591, 2007.

\bibitem{Knudsen-partial}
Kim Knudsen.
\newblock The {C}alder\'{o}n problem with partial data for less smooth
  conductivities.
\newblock {\em Comm. Partial Differential Equations}, 31(1-3):57--71, 2006.

\bibitem{Krupchyk-partial}
Katya Krupchyk and Gunther Uhlmann.
\newblock The {C}alder\'{o}n problem with partial data for conductivities with
  3/2 derivatives.
\newblock {\em Comm. Math. Phys.}, 348(1):185--219, 2016.

\bibitem{lai-partial}
Ru-Yu Lai.
\newblock Stability estimates for the inverse boundary value problem by partial
  {C}auchy data.
\newblock {\em Math. Methods Appl. Sci.}, 38(8):1568--1581, 2015.

\bibitem{recon-2d}
George Lytle, Peter Perry, and Samuli Siltanen.
\newblock Nachman's reconstruction method for the calderon problem with
  discontinuous conductivities.
\newblock {\em Inverse Problems}, 2019.

\bibitem{Mandache}
Niculae Mandache.
\newblock Exponential instability in an inverse problem for the schrödinger
  equation.
\newblock {\em Inverse Problems}, 17(5):1435--1444, aug 2001.

\bibitem{mclean}
William~C. McLean.
\newblock {\em {Strongly Elliptic Systems and Boundary Integral Equations}}.
\newblock Cambridge University Press, 2000.

\bibitem{nachman-partial}
Adrian Nachman and Brian Street.
\newblock Reconstruction in the {C}alder\'{o}n problem with partial data.
\newblock {\em Comm. Partial Differential Equations}, 35(2):375--390, 2010.

\bibitem{nachmanrecon}
Adrian~I. Nachman.
\newblock Reconstructions from boundary measurements.
\newblock {\em Annals of Mathematics}, 128(3):531--576, 1988.

\bibitem{Nachman}
Adrian~I. Nachman.
\newblock Global uniqueness for a two-dimensional inverse boundary value
  problem.
\newblock {\em Ann. of Math. (2)}, 143(1):71--96, 1996.

\bibitem{ns14}
Hoai-Minh Nguyen and Daniel Spirn.
\newblock Recovering a potential from cauchy data via complex geometrical
  optics solutions.
\newblock 2014.

\bibitem{ppu03}
Lassi Päivärinta, Alexander Panchenko, and Gunther Uhlmann.
\newblock Complex geometrical optics solutions for lipschitz conductivities.
\newblock {\em Rev. Mat. Iberoamericana}, 19(1):57--72, 03 2003.

\bibitem{Rodriguez-partial}
Casey Rodriguez.
\newblock A partial data result for less regular conductivities in admissible
  geometries.
\newblock {\em Inverse Probl. Imaging}, 10(1):247--262, 2016.

\bibitem{SU}
John Sylvester and Gunther Uhlmann.
\newblock A global uniqueness theorem for an inverse boundary value problem.
\newblock {\em Ann. of Math. (2)}, 125(1):153--169, 1987.

\bibitem{Triebel}
H.~Triebel.
\newblock {\em Theory of Function Spaces}.
\newblock Modern Birkh{\"a}user Classics. Springer Basel, 2010.

\bibitem{Zhang-partial}
Guo Zhang.
\newblock Uniqueness in the {C}alder\'{o}n problem with partial data for less
  smooth conductivities.
\newblock {\em Inverse Problems}, 28(10):105008, 18, 2012.

\end{thebibliography}
\bibliographystyle{plain}

\end{document}